\documentclass[12pt]{amsart}
\usepackage{amsmath,amsfonts,amsthm,amsopn,amssymb,mathtools,stmaryrd}
\usepackage{cite,marginnote}
\usepackage{pdfsync}

\usepackage{color,enumitem,graphicx}
\usepackage[colorlinks=true,urlcolor=blue,
citecolor=red,linkcolor=blue,linktocpage,pdfpagelabels,
bookmarksnumbered,bookmarksopen]{hyperref}
\usepackage[english]{babel}

\usepackage[left=2.75cm,right=2.75cm,top=3cm,bottom=3cm]{geometry}

\usepackage[hyperpageref]{backref}\usepackage[]{xcolor}

\newcommand{\Hugo}[1]{{\color{green!50!black} #1}}  

\def\sideremark#1{\ifvmode\leavevmode\fi\vadjust{\vbox to0pt{\vss
 \hbox to 0pt{\hskip\hsize\hskip1em
 \vbox{\hsize2.1cm\tiny\raggedright\pretolerance10000
  \noindent #1\hfill}\hss}\vbox to15pt{\vfil}\vss}}}%

\numberwithin{equation}{section}

\makeindex
\newtheorem{theorem}{Theorem}[section]
\newtheorem{proposition}[theorem]{Proposition}
\newtheorem{lemma}[theorem]{Lemma}
\newtheorem{remark}[theorem]{Remark}
\newtheorem{example}[theorem]{Example}
\newtheorem{corollary}[theorem]{Corollary}
\newtheorem{definition}[theorem]{Definition}

\newcommand{\ud}{\mathrm{d}}
\newcommand{\RN}{\mathbb R^N}

\newcommand{\iy}{\infty}

\newcommand{\s}{\section}

\newcommand{\DD}{\Delta}
\newcommand{\g}{\gamma}
\newcommand{\G}{\Gamma}
\newcommand{\na}{\nabla}

\newcommand{\la}{\lambda}

\newcommand{\R}{\mathbb R}
\newcommand{\al}{\alpha}

\newcommand{\ti}{\tilde}

\newcommand{\re}[1]{(\ref{#1})}

\newcommand{\rg}{\rightarrow}

\newcommand{\lan}{\langle}
\newcommand{\ran}{\rangle}
\newcommand{\e}{\varepsilon}
\newcommand{\vp}{\varphi}

\newcommand{\lab}{\label}
\newcommand{\bt}{\begin{theorem}}
\newcommand{\et}{\end{theorem}}
\newcommand{\bl}{\begin{lemma}}
\newcommand{\el}{\end{lemma}}
\newcommand{\bd}{\begin{definition}}
\newcommand{\ed}{\end{definition}}
\newcommand{\bc}{\begin{corollary}}
\newcommand{\ec}{\end{corollary}}
\newcommand{\bp}{\begin{proof}}
\newcommand{\ep}{\end{proof}}
\newcommand{\bx}{\begin{example}}
\newcommand{\ex}{\end{example}}
\newcommand{\bi}{\begin{exercise}}
\newcommand{\ei}{\end{exercise}}
\newcommand{\bo}{\begin{proposition}}
\newcommand{\eo}{\end{proposition}}
\newcommand{\br}{\begin{remark}}
\newcommand{\er}{\end{remark}}
\newcommand{\be}{\begin{equation}}
\newcommand{\ee}{\end{equation}}
\newcommand{\ba}{\begin{align}}
\newcommand{\ea}{\end{align}}
\newcommand{\bn}{\begin{enumerate}}
\newcommand{\en}{\end{enumerate}}
\newcommand{\bg}{\begin{align*}}
\newcommand{\bcs}{\begin{cases}}
\newcommand{\ecs}{\end{cases}}

\newcommand{\pl}{\partial}

\newcommand{\bean}{\begin{eqnarray*}}
\newcommand{\eean}{\end{eqnarray*}}

\renewcommand{\leq}{\leqslant}
\renewcommand{\geq}{\geqslant}

\title[Bose fluids and positive solutions to systems in dimension two]{Bose fluids and positive solutions to weakly coupled systems with critical growth in dimension two}

\author[D.~Cassani]{Daniele Cassani$^\text{1}$}
\author[H.~Tavares]{Hugo Tavares$^\text{2}$}
\author[J.~Zhang]{Jianjun Zhang$^\text{3}$}

\address[D. Cassani]{\newline\indent Dip. di Scienza e Alta Tecnologia
\newline\indent
Universit\`{a} degli Studi dell'Insubria
\newline\indent and
\newline\indent RISM--Riemann International School of Mathematics
\newline\indent Villa Toeplitz, Via G.B. Vico, 46 -- 21100 Varese}
\email{\href{mailto:Daniele.Cassani@uninsubria.it}{Daniele.Cassani@uninsubria.it}}

\address[H.~Tavares]{\newline\indent CMAFcIO, Departamento de Matem\'atica
\newline\indent
Faculdade de Ci\^encias da Universidade de Lisboa
\newline\indent Edif\'icio C6, Piso 1, Campo Grande 1749-016 Lisboa, Portugal}
\email{\href{mailto:hrtavares@ciencias.ulisboa.pt}{hrtavares@ciencias.ulisboa.pt}}

\address[J.~Zhang]{\newline\indent College of Mathematics and Statistics
\newline\indent
Chongqing Jiaotong University
\newline\indent
Chongqing 400074, PR China
\newline\indent and
\newline\indent Dip. di Scienza e Alta Tecnologia
\newline\indent
Universit\`{a} degli Studi dell'Insubria
\newline\indent
via Valleggio 11, 22100 Como,Italy}
\email{\href{mailto:zhangjianjun09@tsinghua.org.cn}{zhangjianjun09@tsinghua.org.cn}}

\thanks{(1) Corresponding author: \texttt{daniele.cassani@uninsubria.it}}
\thanks{(2) H.  Tavares  is  partially  supported  by  ERC  Advanced  Grant  2013  n.
339958  ``Complex  Patterns  for  Strongly  Interacting  Dynamical  Systems  -  COMPAT''  and  by
the Portuguese government through FCT - Funda\c c\~ao para a Ci\^encia e a Tecnologia, I.P., both under the project PTDC/MAT-PUR/28686/2017 and through the grant UID/MAT/04561/2013}
\thanks{(3) The third named author was supported by NSFC(No.11871123) and GLOCAL ERC: Ricercatori di successo internazionale per la ricerca lombarda.}

\date{\today}

\begin{document}

\begin{abstract}
We prove, using variational methods, the existence in dimension two of positive vector ground states solutions for the Bose-Einstein type systems
\[
\begin{cases}
-\DD u+\la_1u=\mu_1u(e^{u^2}-1)+\beta v\left(e^{uv}-1\right) \text{ in } \Omega\ ,\\
-\DD v+\la_2v=\mu_2v(e^{v^2}-1)+\beta u\left(e^{uv}-1\right)\text{ in } \Omega\ ,\\
u,v\in H^1_0(\Omega)
\end{cases}
\]
where $\Omega$ is a bounded smooth domain, $\lambda_1,\lambda_2>-\Lambda_1$ (the first eigenvalue of $(-\Delta,H^1_0(\Omega))$, $\mu_1,\mu_2>0$ and $\beta$ is either positive (small or large) or negative (small). The nonlinear interaction between two Bose fluids is assumed to be of critical exponential type in the sense of J.~Moser. For `small' solutions the system is asymptotically equivalent to the corresponding one in higher dimensions with power-like nonlinearities.

\end{abstract}
\maketitle

\s{Introduction and main results}
\noindent

\noindent In this paper we introduce and study the following system
\begin{equation}\label{q1}
\begin{cases}
-\DD u+\la_1u=\mu_1u(e^{u^2}-1)+\beta v\left(e^{uv}-1\right)& \mbox{in}\,\,\,\Omega,\\
-\DD v+\la_2v=\mu_2v(e^{v^2}-1)+\beta u\left(e^{uv}-1\right) & \mbox{in}\,\,\,\Omega,\\
u,v>0\,\,\,\mbox{in}\,\,\,\Omega,\,\,\,\,\,\,u=v=0\,\,\mbox{on}\,\,\,\pl\Omega,
\end{cases} \end{equation}
where $\Omega\subset\R^2$ is a bounded domain with sufficiently smooth boundary,
$\la_1,\la_2> -\Lambda_1$, with $\Lambda_1=\Lambda_1(\Omega)$ the first eigenvalue of $(-\DD, H_0^1(\Omega))$. The constants $\lambda_i$ represent external Schr\"odinger potentials. The parameters $\mu_i$ take into account the nonlinear interaction due to a single component of the system, and we will consider the focusing case  $\mu_1,\mu_2>0$, whereas $\beta\neq 0$ has the effect of tuning the interaction between different components. When $\beta<0$, the interaction is of cooperative type, while for $\beta<0$ is competitive. In system \eqref{q1} the two equations are weakly coupled in the sense that $u\equiv 0$ does not necessarily imply $v\equiv 0$. Moreover, this class of systems is called of gradient or potential type as the right hand side of \eqref{q1} turns out to be the gradient of a potential function which in this case is given by
\begin{equation}\label{eq:PotH}H(u,v)=\frac{\mu_1}{2}(e^{u^2}-1-u^2)+\frac{\mu_2}{2}(e^{v^2}-1-v^2)+\beta(e^{uv}-1-uv).
\end{equation}

\noindent In 1924, S.N.~Bose \cite{BOSE} discovered the expression for the statistical function of a gas of particles having integer spin (which are then called bosons).
One year later, A.~Einstein \cite{EINSTEIN} applied the results obtained by Bose to a gas of bosons at low temperature, discovering that they can condensate, in the sense that a large fraction of them can
occupy the fundamental state (a fact which is forbidden to fermions, particles with non integer spin, because of the Pauli's exclusion principle: among a huge number of consequences, in this way Einstein showed that quantum mechanics phenomena
are not necessarily microscopic). This phenomenon is known as Bose-Einstein condensation, and the corresponding system, usually
realized by weak interacting bosonic atoms, is called Bose-Einstein condensate (BEC) and serves as prototype for a Bose fluid. More general Bose fluids can be obtained changing the interaction among the Bosons.\\
Usually, Bose fluids with contact interaction terms are described by an operator field $\hat \psi$ governed by an Hamiltonian of the form (in two dimensions)
\begin{align}
 \hat H=\int dx^2 \left( \hat \psi(x)^\dagger \left( -\frac {\hbar^2}{2m} \Delta \right) \hat \psi(x) +\frac {g_n}n \hat \psi(x)^{\dagger n} \hat \psi(x)^n -\mu \hat \psi(x)^\dagger \hat \psi(x) \right),
\end{align}
where $\mu$ is the chemical potential and $^\dagger$ denotes the Hermitian conjugate. In most cases, ``vacua'' configurations are represented by classical solutions (thus regardless of quantum interaction effects), which correspond to $\mathbb C$-valued solutions of the equations of motion, defined by the stationary point
of the Hamiltonian functional. Here $g_n$ is the coupling constant of the model of order $n$. For $n=2$ one gets the Bogoliubov equation \cite{sergio}, for higher $n$ one obtains other phenomenological models. There is no reason a priory
to consider only one monomial for the contact interaction term; polynomial interaction or even power series can be considered as well. However, in general the main difficulty in considering such more general interactions is that one cannot treat those solutions perturbatively, as the number of coupling coefficients in front of the nonlinear terms becomes infinite. 
\noindent From the Physics point of view, system \eqref{q1} describes vacua configurations of two Bose fluids with non-polynomial, actually exponential, self and reciprocal contact interactions with coupling constants $\mu_i$ and $\beta$, and chemical potentials $\lambda_i$, $i=1,2$. 

\noindent For $i=1,2$, consider the following single equation
\be\lab{lm1}
-\DD u+\la_iu=\mu_iu\left(e^{u^2}-1\right)\,\,\mbox{in}\,\,\,\Omega,\,\, u\in H_0^1(\Omega),
\ee
whose corresponding functional is given by $J_{\lambda_i,\mu_i}:H^1_0(\Omega)\to \R$,
$$
J_{\la_i,\mu_i}(u)=\frac{1}{2}\int_{\Omega}\left(|\na u|^2+\la_i u^2\right)\,\ud x-\frac{\mu_i}{2}\int_{\Omega}\left(e^{u^2}-1-u^2\right)\,\ud x.
$$(the functional is well defined by Lemma \ref{lemma:Moser} below). Notice that if $u,v$ solve the single equation \eqref{lm1} for $i=1,2$ repectively, then $(u,0)$ and $(0,v)$ solve system \eqref{q1}. This motivates the following.
\bd
We say that $(u,v)$ is a semitrivial solution of \re{q1} if  $(u,v)$ satisfies \re{q1} and  $u\not\equiv0,v\equiv0$ or $u\equiv0,v\not\equiv0$. On the other hand, $(u,v)$ is called a (nontrivial) vector solution if $(u,v)$ is a solution of \re{q1} with $u\not\equiv0,v\not\equiv0$. A pair $(u,v)$ is called positive if $u>0,v>0$ in $\Omega$.
\ed
\noindent By \cite[Theorem 1.3]{FMR}, problem \re{lm1} admits a mountain pass solution $u_{\la_i,\mu_i}$($u_i$ for short) with minimal energy $E_i<2\pi$, where
\begin{equation}\label{eq:les:Ei}
E_i=\inf\{ J_{\lambda_i,\mu_i}(u):\ u\in H^1_0(\Omega)\setminus\{0\},\ J_{\lambda_i,\mu_i}'(u)=0\}.
\end{equation}
namely a \emph{ground state} solution. Moreover, without loss of generality, we may assume that $u_i$ is positive. Next we define a few quantities which will be crucial in what follows. Set
\begin{equation}\label{eq:beta12}
\beta_1:=\mu_1\frac{\int_{\Omega}u_1^2\left(e^{u_1^2}-1\right)\,\ud x}{\int_{\Omega}u_1^2u_2^2\,\ud x},\qquad \beta_2:=\mu_2\frac{\int_{\Omega}u_2^2\left(e^{u_2^2}-1\right)\,\ud x}{\int_{\Omega}u_1^2u_2^2\,\ud x},
\end{equation}
and
\begin{equation}\label{eq:beta34}
\beta_3:=\mathcal{S}_4\min\left\{\frac{1}{2},\frac{\la_1+\Lambda_1}{2\Lambda_1}\right\}\sqrt{\frac{\mu_2}{E_1+E_2}},\qquad
\beta_4:=\mathcal{S}_4\min\left\{\frac{1}{2},\frac{\la_2+\Lambda_1}{2\Lambda_1}\right\}\sqrt{\frac{\mu_1}{E_1+E_2}},
\end{equation}
where $\mathcal{S}_4$ is the best Sobolev constant of the embedding $H_0^1(\Omega)\hookrightarrow L^4(\Omega)$, i. e.
$$
\mathcal{S}_4:=\inf_{u\in H_0^1(\Omega)\setminus\{0\}}\frac{\int_{\Omega}|\na u|^2\,\ud x}{\left(\int_{\Omega}u^4\,\ud x\right)^{\frac12}}\ .
$$
Let $\beta^\ast=\min\{\beta_i:i=1,2,3,4\}$ and let $\beta^{\ast\ast}>0$ be as in Corollary \ref{tl} (see Section \ref{sec3} below). Then, our first two main results deal respectively with the case $\beta>0$ small and large (respectively weak and strong cooperation) and read as follows:
\bt\lab{Th1} Let $\la_1,\la_2> -\Lambda_1$ and
$$
\beta_0=\min\{\sqrt{\mu_1\mu_2},\beta^\ast,\beta^{\ast\ast}\}
\ .$$
Then, for any $\beta\in(0,\beta_0)$, \eqref{q1} admits a positive vector ground state solution $(u_\beta,v_\beta)\in H^1_0(\Omega)\times H_0^1(\Omega)$. Moreover, up to a subsequence, as $\beta\rg0$, $u_\beta\rg u$ and $v_\beta\rg v$ strongly in $H_0^1(\Omega)$, where $u,v$ are ground state solutions of \re{lm1} with $i=1,2$ respectively. \et
\noindent Set
$$
\beta_5:=\mu_1\frac{\int_{\Omega}u_1^2\left(e^{u_1^2}-1\right)\,\ud x}{\int_{\Omega}u_1^4\,\ud x}\quad\text{ and }\quad  \beta_6:=\mu_2\frac{\int_{\Omega}u_2^2\left(e^{u_2^2}-1\right)\,\ud x}{\int_{\Omega}u_2^4\,\ud x}.
$$

\bt\lab{Th2} Let $\la_1,\la_2> -\Lambda_1$ and
\begin{equation}\label{eq:barbeta0}
\bar{\beta}_0=4\frac{\max\{E_1\beta_5, E_2\beta_6\}}{\min\{E_1,E_2\}}>0.
\end{equation}
Then, for any $\beta\ge\bar{\beta}_0$, \eqref{q1} admits a positive vector ground state solution $(u_\beta,v_\beta)$. Moreover, as $\beta\rg+\iy$, $u_\beta\rg 0$ and $v_\beta\rg 0$ strongly in $H_0^1(\Omega)$. \et
\br
As the value $\beta^*$ depends on the unknown embedding constant $\mathcal{S}_4$, we will establish bounds for $\beta^*$ in Section \ref{sec_lowerbeta*}.
\er

\noindent Our third main result concerns the weak competitive case (small $\beta<0$). We prove the following

\bt\lab{Th3} Let $\la_1,\la_2> -\Lambda_1$ and $\beta<0$. Then, for $|\beta|$ sufficiently small, \eqref{q1} admits a positive vector solution $(u_\beta,v_\beta)\in H^1_0(\Omega)\times H_0^1(\Omega)$. Moreover, as $\beta\rg0$, $u_\beta\rg u$ and $v_\beta\rg v$ strongly in $H_0^1(\Omega)$, where $u,v$ are ground state solutions of \re{lm1} with $i=1,2$ respectively. \et

\br
From Theorems \ref{Th1} and \ref{Th2} one has existence of solutions for $\beta$ in a neighborhood of zero as well as of infinity. We mention that for power-type nonlinearities, it turns out that actually `reasonable' solutions for $\min \{\mu_1,\mu_2\}\leq \beta\leq \max \{\mu_1,\mu_2\}$ do not exists, see \cite{sirakov} for the cubic case. Moreover, as $\beta\to -\infty$ one expects the appearance of segregation phenomena, in the sense that solutions concentrate on disjoint support, see for instance \cite{CTV1,CTV2,WeiWeth,NTTV,SoaveZilio}, or \cite{STTZ} for a recent survey on the subject. However, in dimension two and for exponential nonlinearities, analogous results seem to be out of reach at the moment, in particular due to the difficulty of obtaining existence results for $\beta$ negative.
\er

\noindent Thanks to Theorem \ref{Th2}, as $\beta\to +\infty$ one has that system \eqref{q1} is asymptotically equivalent to the following system
\begin{equation}\label{BEH}
\begin{cases}
-\DD u+\la_1u=\mu_1u^3+\beta uv^2&\mbox{in}\,\,\,\Omega,\\
-\DD v+\la_2v=\mu_2v^3+\beta u^2v &\mbox{in}\,\,\,\Omega,\\
u,v>0\,\,\,\mbox{in}\,\,\,\Omega,\,\,\,\,\,\,u=v=0\,\,\mbox{on}\,\,\,\pl\Omega,
\end{cases}
\end{equation}
since $u_\beta(e^{u_\beta^2}-1)\sim u_\beta^3$, $v_\beta(e^{v_\beta^2}-1)\sim v_\beta^3$ and $e^{u_\beta v_\beta}-1\sim u_\beta v_\beta$, as $\beta\to +\infty$.

\noindent  In recent years, existence results fo system \eqref{BEH} have been largely investigated in a series of papers, see for instance \cite{AmbrosettiColorado, ChenZou, LinWei1,LinWei2,MaiaMontefuscoPellacci,sirakov,Mandel} (see also the introduction of \cite{SoaveTavares} for more details and results regarding systems with three or more equations). In the cooperative case $\beta>0$ this has applications to nonlinear optics \cite{Phys1, Phys2}, while in the competitive case $\beta<0$ this is related to Bose-Einstein condensates \cite{Timm}.

\noindent At the best of our knowledge, the case of dimension $N=2$ with exponential nonlinearities has not yet been settled, and this paper is a first step in this direction, in the hope of stimulating further research, see \cite{DJ17} and references therein for related results.

\noindent It is well known that, coming from higher dimension $N\geq 3$ to $N=2$, nonlinear phenomena change dramatically from allowing power-like growth at infinity up to the limiting case of exponential growth.

\noindent From the point of view of functional analysis, this clue can be motivated in terms of Sobolev embeddings for which dimension two is a borderline case. In fact as $N=2$, functions with membership in the energy space $H^1_0$ turn out to belong to $L^p$ for any finite order of integrability $1\leq p<\infty$. As established by the Pohozaev-Trudinger-Moser inequality, the maximal degree of summability is of exponential type, see Section \ref{sec2}.

\noindent In the variational framework, when passing from power-like to exponential nonlinearities, extra difficulties appear due to the lack of homogeneity and the presence of infinite series of powers-like nonlinear interactions. Indeed, the main difficulty in considering \eqref{q1}, as the two dimensional counterpart of \eqref{BEH}, is that one misses the cubic homogeneity which is manifest in the right hand side of \eqref{BEH}: as we are going to see this turns out to be a main obstruction for this class of systems.

\medbreak

\noindent The paper is organized as follows. In Section \ref{sec2} we provide the variational setting of the problem, recall some properties of the single equation case \eqref{lm1}, and give a lower estimate of the constant $\beta^*$ appearing in Theorem \ref{Th1}, depending only on the parameters of the system, the first Dirichlet eigenvalue, and on the critical Moser energy level. Section \ref{sec3} deals with the case of $\beta>0$ small (weak cooperation), that is with the proof of Theorem \ref{Th1}. Section \ref{sec4} deals with the proof of Theorem \ref{Th2} (case $\beta>0$ large, strong cooperation). Finally, Section \ref{sec5} is concerned with Theorem \ref{Th3}, were existence of positive vector solutions is proved for $\beta<0$ small (weak competition) by means of  a perturbation argument.
\noindent Throughout the paper, we will denote the standard $L^p$--norms simply by $\|\cdot\|_p$, for $p\geq 1$.

\s{Preliminaries and variational setting}\label{sec2}

\subsection{The energy functional}
Since we are concerned with positive solutions to system \re{q1}, in what follows we consider the system
\be\lab{q2} \left\{
\begin{array}{ll}
-\DD u+\la_1u=H_u(u,v)\,\,\,\mbox{in}\,\,\,\Omega,\\
-\DD v+\la_2v=H_v(u,v)\,\,\,\mbox{in}\,\,\,\Omega,\\
u,v\in H_0^1(\Omega),
\end{array}
\right.
\ee
where
\begin{align*}
H(u,v)&=\frac{\mu_1}{2}G(u,u)+\beta G(u,v)+\frac{\mu_2}{2}G(v,v)
\end{align*}
and
$$
G(u,v)=e^{|uv|}-1-|uv|.
$$
(recall \eqref{eq:PotH}). Let $X=H_0^1(\Omega)\times H_0^1(\Omega)$ and define the energy functional associated with system \re{q1} as follows
$$
I(u,v)=\frac12\int_{\Omega}(|\na u|^2+|\na v|^2+\la_1u^2+\la_2v^2)-\int_{\Omega}H(u,v),\,\,(u,v)\in X.
$$
This is well defined and of class $C^2(X)$- see Proposition \ref{prop:Iwelldef} below - and its derivative is given by
\begin{align*}
\lan I'(u,v),(\vp,\phi)\ran=&\int_{\Omega}(\na u\na\vp+\na v\na\phi+\la_1u\vp+\la_2v\phi)\\
&-\int_{\Omega}H_u(u,v)\vp+H_v(u,v)\phi,\,\,(u,v), (\vp,\phi)\in X.
\end{align*}
Observe that $e^{|x|}-1-|x|\in C^2(\R)$, $H\in C^2(\R^2)$. In order to show that $I$ is well defined in $X$, we first give the following elementary inequalities. Next, we recall some standard facts.
\bl\lab{bd1} For any $x,y\in\R^+$, the following holds:
\begin{itemize}
\item [{\rm(i)}]  $(e^{xy}-1)^2\le(e^{x^2}-1)(e^{y^2}-1)$;
\item [{\rm(ii)}]  $(e^{xy}-1-xy)^2\le(e^{x^2}-1-x^2)(e^{y^2}-1-y^2)$;
\item [{\rm(iii)}]  $0\le2(e^{x}-1-x)\le x(e^{x}-1)$.
\end{itemize}
\el
\bp
For any fixed $y>0$, let
$$
f(x)=\frac{(e^{x^2}-1)(e^{y^2}-1)}{(e^{xy}-1)^2},\,\,x\in[y,\iy).
$$
Obviously,
$$
f'(x)=\frac{2(e^{y^2}-1)}{x(e^{xy}-1)^3}g(x),
$$
where $g(x)=x^2e^{x^2}(e^{xy}-1)-xye^{xy}(e^{x^2}-1)$. Let $h(t)=\frac{te^t}{e^t-1}$, then
$$
h'(t)=\frac{e^t(e^{t}-1-t)}{(e^{t}-1)^2}>0,\,\,t>0.
$$
It follows that $h(x^2)>h(xy)$ and so $g(x)>0$ for $x>y>0$. This implies that $f(x)$ is increasing in $(y,\iy)$. Noting that $f(y)=1$, the assertion $(i)$ is proved.
\vskip0.1in
Now, we prove $(ii)$. For any fixed $y>0$, let
$$
\ti{f}(x)=\frac{(e^{x^2}-1-x^2)(e^{y^2}-1-y^2)}{(e^{xy}-1-xy)^2},\,\,x\in[y,\iy).
$$
Then
$$
\ti{f}'(x)=\frac{2(e^{y^2}-1-y^2)}{x(e^{xy}-1-xy)^3}\ti{g}(x),
$$
where $\tilde g(x)=x^2(e^{x^2}-1)(e^{xy}-1-xy)-xy(e^{xy}-1)(e^{x^2}-1-x^2)$. Let $\ti{h}(t)=\frac{t(e^t-1)}{e^t-1-t}$, then
$$
\ti{h}'(t)=\frac{(e^{t}-1)^2-t^2e^t}{(e^{t}-1-t)^2},\,\,t>0.
$$
It is straightforward to check that $(e^{t}-1)^2-t^2e^t>0$ for any $t>0$. Then $\ti{h}'(t)>0$ for any $t>0$ and $\ti{h}(x^2)>\ti{h}(xy)$ for $x>y>0$. This implies that $g(x)>0$ for $x>y>0$ and $f(x)$ is increasing in $(y,\iy)$. Noting that $\ti{f}(y)=1$, the assertion $(ii)$ is concluded.
\vskip0.1in

 As for (iii), this follows by direct inspection:
\begin{align*}
2(e^x-1-x)=2\sum_{n\geq 2} \frac{x^n}{n!}=\sum_{n\geq 1}\frac{2x^{n+1}}{(n+1)!}\leq \sum_{n\geq 1} \frac{x^{n+1}}{n!}=x(e^x-1),
\end{align*}
since $2/(n+1)\leq 1$ for $n\geq 1$.
\ep
\bl\lab{AR}
For any $(x,y)\in\R^2\setminus\{0\}$,
$$
xH_x(x,y)+yH_y(x,y)\ge4H(x,y)>0\,\,\,\mbox{if}\,\,\beta>0.
$$
\el
\begin{proof}
Notice that
$$
xH_x(x,y)=\mu_1x^2(e^{x^2}-1)+\beta|xy|(e^{|xy|}-1),
$$
and
$$
yH_y(x,y)=\mu_2y^2(e^{y^2}-1)+\beta|xy|(e^{|xy|}-1).
$$
Then, by Lemma \ref{bd1}-(iii) and since $\beta>0$,
$$
xH_x(x,y)\ge 2\mu_1(e^{x^2}-1-x^2)+2\beta(e^{|xy|}-1-|xy|),
$$
and
$$
yH_y(x,y)\ge 2\mu_2(e^{y^2}-1-y^2)+2\beta(e^{|xy|}-1-|xy|).
$$
So $xH_x(x,y)+yH_y(x,y)\ge4H(x,y)$.
\end{proof}

\noindent Let us recall the well known Pohozaev-Trudinger-Moser inequality, in the form which is due to J.~Moser.
\bl\lab{Moser}{\rm\cite[Moser]{Moser}}\label{lemma:Moser} For any $u\in H_0^1(\Omega)$ and $\alpha>0$, $\mathrm{e}^{\alpha u^2}\in L^1(\Omega )$. Moreover,
\begin{equation}\label{PTM1}
\mathop{\sup_{u \in H^1_0(\Omega)}}_{\|\nabla u\|_{2}\leq 1}\int_{\Omega} \mathrm{e}^{\alpha
 u^2}\ \ud x = c(\alpha)|\Omega|
\end{equation}
with
\[
c(\alpha) <\infty   \quad \mbox{if} \quad \alpha \leq 4\pi,\qquad c(\alpha)= +\infty \quad  \mbox{if} \quad \alpha > 4\pi.
\]

\el

\noindent The following result can be found in the proof of \cite[Theorem I.6]{Lions}, see page 197 therein.
\bl\lab{Lions}{\rm\cite[P. L. Lions]{Lions}}
Assume that $\{u_n\}\subset H_0^1(\Omega)$ with $\|\na u_n\|_2\leq 1$, $u_n\rightharpoonup u\not\equiv0$ weakly in $H_0^1(\Omega)$. Then
$$
\sup_n\int_{\Omega}e^{4p\pi u_n^2}\,\ud x<\iy,\,\,\,\mbox{for any }0<p<(1-\|\na u\|_2^2)^{-1}.
$$
\el

\bo  \label{prop:Iwelldef}
For any $(u,v)\in X$, $I(u,v)$ and $I'(u,v)$ are well defined and $I\in C^2(X)$.
\eo
\bp
For any $(u,v)\in X$, by Lemma \ref{Moser} we know that $\int_{\Omega}e^{u^2}<\iy$ and $\int_{\Omega}e^{v^2}<\iy$. So $\int_{\Omega}G(u,u)<\iy$ and $\int_{\Omega}G(v,v)<\iy$. By Lemma \ref{bd1} and Schwarz's inequality, we get
$$
0\le\int_{\Omega}G(u,v)\le\int_{\Omega}\sqrt{G(u,u)G(v,v)}\le\left(\int_{\Omega}G(u,u)\right)^{1/2}\left(\int_{\Omega}G(v,v)\right)^{1/2}<\iy.
$$
It follows that $|I(u,v)|<\iy$. For any $(\vp,\phi)\in X$, to show $\lan I'(u,v),(\vp,\phi)\ran$ is well defined, it suffices to prove
$$
\left|\int_{\Omega}H_u(u,v)\vp+H_v(u,v)\phi\right|<\iy.
$$
In fact, noting that $H_0^1(\Omega)\hookrightarrow L^p(\Omega)$ for any $p\ge1$ (we are working in dimension two),
\begin{align*}
\left|\int_{\Omega}H_u(u,v)\vp\right|&\le \int_{\Omega}\mu_1|u\vp|(e^{u^2}-1)+|\beta||v\vp|(e^{|uv|}-1)\\
&\le \mu_1\left(\int_{\Omega}u^4\right)^{1/4}\left(\int_{\Omega}\vp^4\right)^{1/4}\left(\int_{\Omega}(e^{2u^2}+1)\right)^{1/2}\\
&\,\,\,\,+\beta\left(\int_{\Omega}v^4\right)^{1/4}\left(\int_{\Omega}\vp^4\right)^{1/4}\left(\int_{\Omega}(e^{|uv|}-1)^2\right)^{1/2}.
\end{align*}
By Lemma \ref{bd1}-(i), Lemma \ref{Moser} and Young's inequality,
$$
\int_{\Omega}(e^{|uv|}-1)^2\le\int_{\Omega}(e^{u^2}-1)(e^{v^2}-1)\le \int_{\Omega}(e^{2u^2}+e^{2v^2})<\iy.
$$
Then $\left|\int_{\Omega}H_u(u,v)\vp\right|<\iy$. Similarly, we have $\left|\int_{\Omega}H_v(u,v)\phi\right|<\iy.$ Thus, $I'(u,v)$ is well defined. Finally, by a standard argument, one can verify that $I\in C^2(X)$.
\ep
\subsection{The limit problem}\label{subsec:limitproblem}
For any $\la>-\Lambda_1$ and $\mu>0$, consider the following problem
\be\lab{limit}
-\DD u+\la u=\mu u\left(e^{u^2}-1\right)\,\,\mbox{in}\,\,\,\Omega,\,\, u\in H_0^1(\Omega).
\ee

\noindent By \cite[Theorem 1.3]{FMR}, problem \re{limit} admits a mountain pass solution $u_{\la,\mu}$. In fact, one can check that $u_{\la,\mu}$ is also a ground state solution, i.e., it minimizes the energy among the set of all nontrivial solutions. To be more precise, define the associated energy functional by
$$
J_{\la,\mu}(u)=\frac{1}{2}\int_{\Omega}\left(|\na u|^2+\la u^2\right)\,\ud x-\frac{\mu}{2}\int_{\Omega}\left(e^{u^2}-1-u^2\right)\,\ud x,
$$
denote the ground state level by
$$
E_{\la,\mu}:=\min\{ J_{\lambda,\mu}(u):\ u\in H^1_0(\Omega)\setminus\{0\},\ J'_{\lambda,\mu}(u)=0 \},
$$
and the set of ground state solutions of  \re{limit} by
\begin{equation}\label{eq:defSlambdamu}
S_{\la,\mu}:=\{ u\in H^1_0(\Omega)\setminus\{0\}:\ J'_{\lambda,\mu}(u)=0,\ J_{\lambda,\mu}(u)=E_{\la,\mu} \}.
\end{equation}
Then $u_{\lambda,\mu}\in S_{\la,\mu}$, which is nonempty. Moreover, $E_{\la,\mu}\in (0,2\pi)$ and
$$
E_{\la,\mu}=\inf_{u\in H_0^1(\Omega)\setminus\{0\}}\max_{t\ge0}J_{\la,\mu}(tu)=\inf_{u\in\mathcal{N}_{\la,\mu}}J_{\la,\mu}(u).
$$
where
$$
\mathcal{N}_{\la,\mu}:=\left\{u\in H_0^1(\Omega)\setminus\{0\}:\int_{\Omega}\left(|\na u|^2+\la u^2\right)\,\ud x=\mu\int_{\Omega}u^2\left(e^{u^2}-1\right)\,\ud x\right\},
$$
is the Nehari manifold. We refer to \cite{FMR} for the details.

\noindent Next we recall an Adachi-Tanaka type inequality due to Cassani-Sani-Tarsi \cite{Cassani}.

\bl {\rm\cite[Theorem 1.2]{Cassani}}\lab{at} For all $u\in H^1(\R^2)$ with $\|\na u\|_2\le1$, then the following holds
$$
\int_{\R^2}\left(e^{\gamma u^2}-1\right)\,\ud x\le\frac{\gamma d_{4\pi}}{4\pi-\gamma}\|u\|_2^2,\qquad \mbox{if}\,\,\gamma<4\pi,
$$
where
\begin{equation}\label{eq:criticalMoser}
d_{4\pi}:=\sup_{\stackrel{u\in H^1(\R^2)}{\|\na u\|_2^2+\|u\|_2^2\le1}}\int_{\R^2}\left(e^{4\pi u^2}-1\right)\,\ud x<\infty.
\end{equation}
\el
\br
From \cite[Theorem 1.2]{Cassani} we know that there exists $C>0$ such that for all $u\in H^1(\R^2)$ with $\|\na u\|_2\le1$,
$$
\int_{\R^2}\left(e^{\gamma u^2}-1\right)\,\ud x\le\frac{C}{1-\gamma/4\pi}\|u\|_2^2,\,\,\mbox{if}\,\,\gamma<4\pi.
$$
A close inspection of the proof (see page 4248 therein) shows that $C=\gamma d_{4\pi}/4\pi$. The fact that $d_{4\pi}<\infty$ is shown in \cite{Ruf}.
\er

\bl\lab{atn}
For all $u\in H^1(\R^2)$ with $\|\na u\|_2\le1$, the following holds
$$
\int_{\R^2}u^2\left(e^{\gamma u^2}-1\right)\,\ud x\le C(\gamma)\|u\|_4^4,\,\,\mbox{if}\,\,\gamma<4\pi,
$$
where
$$
C(\gamma):=4\pi\max\left\{e^{\frac{\gamma}{4\pi}}-1,\frac{16\pi(4\pi+\gamma)}{(4\pi-\gamma)^2}e^{\frac{2\gamma}{4\pi-\gamma}}\right\}.
$$
\el
\bp
The proof is similar to \cite{AT,Moser}. Without loss of generality, we only consider the functions $u\in H^1(\R^2)$ with $\|\na u\|_2\le1$, which are nonnegative, compactly supported, radially symmetric and $u(|x|)$ is decreasing in $r=|x|$. Let
$$
w(t)=2\pi^{1/2}u(r),\,\,r=|x|=e^{-t/2}.
$$
Then $w(t),w'(t)\ge0$ for all $t\in\R$ and $w(t_0)=0$ for some $t_0$. Moreover, $\int_{\R^2}|\na u|^2\,\ud x=\int_{-\iy}^{+\iy}|w'(t)|^2\,\ud t$,
$$
\int_{\R^2}u^2\left(e^{\gamma u^2}-1\right)\,\ud x=\frac{1}{4}\int_{-\iy}^{+\iy}w^2\left(e^{\frac{\gamma}{4\pi}w^2}-1\right)e^{-t}\,\ud t,
$$
and
$$
\int_{\R^2}u^4\,\ud x=\frac{1}{16\pi}\int_{-\iy}^{+\iy}w^4e^{-t}\,\ud t.
$$
Let $T_0:=\sup\{t\in\R: w(t)\le1\}$, then $w(T_0)=1$ if $T_0<\iy$. Observing that
$$
e^{\frac{\gamma}{4\pi}s}-1\le\left(e^{\frac{\gamma}{4\pi}}-1\right)s,\,\,\forall s\in[0,1],
$$
we have
\be\lab{m1}
\int_{-\iy}^{T_0}w^2\left(e^{\frac{\gamma}{4\pi}w^2}-1\right)e^{-t}\,\ud t\le\left(e^{\frac{\gamma}{4\pi}}-1\right)\int_{-\iy}^{T_0}w^4e^{-t}\,\ud t.
\ee
On the other hand, if $T_0<\iy$, then as can be seen in \cite[p. 2055]{AT}, for any $\e\in(0,\frac{4\pi}{\gamma}-1)$,
$$
w^2(t)\le(1+\e)(t-T_0)+1+\e^{-1},\,\,\forall t\ge T_0.
$$
So
\begin{align*}
&\int_{T_0}^{+\iy}w^2\left(e^{\frac{\gamma}{4\pi}w^2}-1\right)e^{-t}\,\ud t\\
&\le e^{\frac{\gamma}{4\pi}(1+\e^{-1})-T_0}\int_{T_0}^{+\iy}[(1+\e)(t-T_0)+1+\e^{-1}]e^{\left[\frac{\gamma}{4\pi}(1+\e)-1\right](t-T_0)}\,\ud t\\
&=e^{\frac{\gamma}{4\pi}(1+\e^{-1})-T_0}\left(\frac{1+\e^{-1}}{1-\frac{\gamma}{4\pi}(1+\e)}+\frac{1+\e}{\left[1-\frac{\gamma}{4\pi}(1+\e)\right]^2}\right).
\end{align*}
Noting that $w(t)\ge1$ for $t\ge T_0$, we have $\int_{T_0}^{+\iy}w^4 e^{-t}\,\ud t\ge e^{-T_0}$ and then
\begin{align*}
&
\int_{T_0}^{+\iy}w^2\left(e^{\frac{\gamma}{4\pi}w^2}-1\right)e^{-t}\,\ud t\\
&\le e^{\frac{\gamma}{4\pi}(1+\e^{-1})}\left(\frac{1+\e^{-1}}{1-\frac{\gamma}{4\pi}(1+\e)}+\frac{1+\e}{\left[1-\frac{\gamma}{4\pi}(1+\e)\right]^2}\right)\int_{T_0}^{+\iy}w^4 e^{-t}\,\ud t.
\end{align*}
Taking $\e=\frac{4\pi-\gamma}{4\pi+\gamma}\in(0,\frac{4\pi}{\gamma}-1)$ and let
$$
C(\gamma):=4\pi\max\left\{e^{\frac{\gamma}{4\pi}}-1,
e^{\frac{\gamma}{4\pi}(1+\e^{-1})}\left(\frac{1+\e^{-1}}{1-\frac{\gamma}{4\pi}(1+\e)}+\frac{1+\e}{\left[1-\frac{\gamma}{4\pi}(1+\e)\right]^2}\right)\right\},
$$
then
$$
C(\gamma):=4\pi\max\left\{e^{\frac{\gamma}{4\pi}}-1,\frac{16\pi(4\pi+\gamma)}{(4\pi-\gamma)^2}e^{\frac{2\gamma}{4\pi-\gamma}}\right\},
$$
and
\[
\int_{\R^2}u^2\left(e^{\gamma u^2}-1\right)\,\ud x\le C(\gamma)\|u\|_4^4. \qedhere
\]
\ep

\subsection{A priori estimates}
\bl\lab{priori}
The set $S_{\la,\mu}$ is compact in $H_0^1(\Omega)$ and uniformly bounded in $L^\iy(\Omega)$.
\el
\bp
\vskip0.1in
As in \cite{ZJM}, we can use the Nash-Moser iteration technique (see \cite{Gongbao}) to prove that $u\in L^\iy(\Omega)$. Actually, that same procedure allows one  to pass from $H^1_0(\Omega)$ bounds to $L^\infty(\Omega)$ bounds. In the following, we show the compactness of $S_{\la,\mu}$ which enables us to get the uniform boundedness of $S_{\la,\mu}$ in $L^\iy(\Omega)$.

\noindent Let $\{u_n\}\subset S_{\la,\mu}$. The energy $J_{\lambda,\mu}$ is clearly uniformly bounded on $S_{\lambda,\mu}$, and therefore, by using Lemma \ref{bd1}-(iii), $\lambda>-\Lambda_1$, and Poincar\'e's inequality,
\begin{align*}\lab{wo}
C&\geq J_{\lambda,\mu}(u_n)= J_{\lambda,\mu}(u_n)-\frac{1}{4}\lan J_{\lambda,\mu}'(u_n),u_n\ran\nonumber\\
&=\frac{1}{4}\int_{\Omega}\left(|\na u_n|^2+\la u_n^2\right)\,\ud x+\frac{\mu}{4}\int_{\Omega}\left[u_n^2\left(e^{u_n^2}-1\right)-2\left(e^{u_n^2}-1-u_n^2\right)\right]\,\ud x\\
&\geq \frac{1}{4}\int_{\Omega}\left(|\na u_n|^2+\la u_n^2\right)\,\ud x  =  \frac{1}{4} \int_\Omega u_n^2(e^{u_n^2}-1) \,\ud  x\ge\frac{1}{4}\min\left\{1,\frac{\la+\Lambda_1}{\Lambda_1}\right\}\int_{\Omega}|\na u_n|^2\,\ud x.\nonumber
\end{align*}
Therefore there exists $C>0$ such that for all $n$,
\be\lab{boundl}
\|\na u_n\|_2\leq C\quad\text{ and }\quad \int_{\Omega}u_n^2\left(e^{u_n^2}-1\right)\,\ud x\le C.
\ee

\noindent Moreover,
\be\lab{nee}
\liminf_{n\rg\iy}\|\na u_n\|_2\ge \rho_{\la,\mu}\ge0,
\ee
where
\begin{equation}\label{eq:rho_lamu}
\rho_{\la,\mu}:=\inf_{u\in S_{\la,\mu}}\|\na u\|_2.
\end{equation} In fact, $\rho_{\la,\mu}>0$. Otherwise, if $\rho_{\la,\mu}=0$, then there exists $\{v_n\}\subset S_{\la,\mu}$ with $\|\na v_n\|_2\rg0$ as $n\rg\iy$. It follows from Lemma \ref{atn} that
$$
\liminf_{n\rg\iy}\int_{\Omega}v_n^2\left(e^{v_n^2}-1\right)\,\ud x=0,
$$
which implies by Lemma \ref{bd1}-(iii) that
$$
\liminf_{n\rg\iy}\int_{\Omega}\left(e^{v_n^2}-1-v_n^2\right)\,\ud x=0.
$$
Then we have $\liminf_{n\rg\iy}J_{\la,\mu}(v_n)=0$, which contradicts $E_{\la,\mu}>0$.

\smallbreak

\noindent Without loss of generality, we may assume $u_n\rg u$ weakly in $H_0^1(\Omega)$, strongly in $L^2(\Omega)$ and a.e. in $\Omega$, as $n\rg\iy$. By \cite[Lemma 2.1]{FMR}, we have
that $u_n(e^{u_n^2}-1)\to u(e^{u^2}-1)$ in $L^1(\Omega)$, which yields
\begin{equation}\label{eq:FMR}
\lim_{n\to \infty} \int_\Omega (e^{u_n^2}-1-u_n^2)\,\ud x=\lim_{n\to \infty}\int_\Omega (e^{u^2}-1-u^2)\,\ud x
\end{equation}
\noindent We divide the rest of the proof in three steps:

\noindent \textbf{Step 1.} First we prove that $u\not\equiv0$. Indeed otherwise, if $u=0$ then by \eqref{eq:FMR} we have
$$
\lim_{n\rg\iy}\int_{\Omega}\left(e^{u_n^2}-1-u_n^2\right)\,\ud x=0,
$$
and so
$$
\lim_{n\rg\iy}\int_{\Omega}|\na u_n|^2\,\ud x=2\lim_{n\rg\iy}J_{\la,\mu}(u_n)=2E_{\la,\mu}<4\pi.
$$
Moreover, for some $q>1$ (sufficiently close to 1) such that $2qE_{\la,\mu}<4\pi$, by Lemma \ref{at} with $\gamma=q\|\nabla u_n\|_2^2$, for $n$ large enough one has
\begin{align*}
&\int_{\Omega}\left(e^{q u_n^2}-1\right)\,\ud x=\int_{\Omega}\left(e^{q\|\na u_n\|_2^2\left|\frac{u_n}{\|\na u_n\|_2}\right|^2}-1\right)\,\ud x\le \frac{q\|\nabla u_n\|_2^2 d_{4\pi}}{4\pi-q\|\nabla u_n\|_2^2}\frac{\|u_n\|_2^2}{\|\na u_n\|_2^2},
\end{align*}
which implies, by \re{nee} and since $u_n\rg0$ strongly in $L^2(\Omega)$, that
$$
\lim_{n\rg\iy}\int_{\Omega}\left(e^{q u_n^2}-1\right)\,\ud x=0.
$$
Then
$$
\lim_{n\rg\iy}\int_{\Omega}u_n^2\left(e^{u_n^2}-1\right)\,\ud x\le\lim_{n\rg\iy}\left[\int_{\Omega}u_n^{2p}\,\ud x\right]^{1/p}\left[\int_{\Omega}\left(e^{u_n^2}-1\right)^q\,\ud x\right]^{1/q}=0,
$$
where $p>1$ and $1/p+1/q=1$. So $\|\na u_n\|_2\rg0$, as $n\rg\iy$, which contradicts \re{nee}. Thus, $u\not\equiv0$.

\smallbreak

\noindent \textbf{Step 2.}  Next we show that
\be\lab{uppl}
\int_{\Omega}\left(|\na u|^2+\la u^2\right)\,\ud x=\int_{\Omega}\mu u^2\left(e^{u^2}-1\right)\,\ud x.
\ee
For any $\vp\in C_0^\iy(\Omega)$, by \re{boundl}, we have
$$
\sup_{n}\int_{\Omega}u_n^2(e^{u_n^2}-1)|\vp(x)|\,\ud x<\iy.
$$
Then it follows from \cite[Lemma 2.1]{FMR} that, up to a subsequence, we have
$$
\lim_{n\rg\iy}\int_{\Omega}u_n(e^{u_n^2}-1)\vp(x)\,\ud x=\int_{\Omega}u(e^{u^2}-1)\vp(x)\,\ud x.
$$
Recalling that $u_n\rightharpoonup u$ weakly in $H_0^1(\Omega)$ as $n\rg\iy$ and $u_n$ satisfies
$$
-\DD u_n+\la u_n=\mu u_n(e^{u_n^2}-1),\,\,x\in\Omega,
$$
we get that
$$
-\DD u+\la u=\mu u(e^{u^2}-1),\,\,x\in\Omega,
$$
and so \re{uppl} holds true.

\textbf{Step 3}
\noindent Since $u\not\equiv 0$, by definition of  $E_{\la,\mu}$, the semicontinuity of the norms and \eqref{eq:FMR}, we have
\[
E_{\lambda,\mu}\leq \liminf_{n\to \infty} J_{\lambda,\mu}(u_n)=E_{\lambda,\mu}.
\]
and so $J_{\la,\mu}(u)=E_{\la,\mu}$. Again by \eqref{eq:FMR}, we get that $\|\na u_n\|_2\rg\|\na u\|_2$, as $n\rg\iy$. Thus, $u_n\rg u$ strongly in $H_0^1(\Omega)$ as $n\rg\iy$. That is, $S_{\la,\mu}$ is compact in $H_0^1(\Omega)$.
\ep

\subsection{Lower bounds for  $\beta^*$}\label{sec_lowerbeta*} Let $\beta^\ast=\min\{\beta_i:i=1,2,3,4\}$, where the $\beta_i$'s are as in \eqref{eq:beta12}--\eqref{eq:beta34}. The purpose of this subsection it to give a lower estimate of $\beta^*$ which just depends on $\la_1,\la_2,\mu_1,\mu_2,\Lambda_1$ and the critical Moser energy level $d_{4\pi}$ defined in \eqref{eq:criticalMoser}.  More precisely, we prove the following

\begin{lemma}\label{lemma:estimatebeta*} We have
\begin{align*}
\beta^\ast>\min& \left\{\sqrt{\frac{\mu_1\mu_2}{32}},\sqrt{\frac{\mu_1\mu_2}{32}\frac{\la_1+\Lambda_1}{\Lambda_1}},\frac{e^{-1/3}}{48}\sqrt{\frac{\Lambda_1\mu_2}{\pi d_{4\pi}}(4\pi-1)}\min\left\{1,\frac{\la_1+\Lambda_1}{\Lambda_1}\right\},\right.\\
&\left.\,\,\,\,\,\,\,\sqrt{\frac{\mu_1\mu_2}{32}\frac{\la_2+\Lambda_1}{\Lambda_1}},\frac{e^{-1/3}}{48}\sqrt{\frac{\Lambda_1\mu_1}{\pi d_{4\pi}}(4\pi-1)}\min\left\{1,\frac{\la_2+\Lambda_1}{\Lambda_1}\right\},\right.\\
&\left.\,\,\,\,\,\,\,\sqrt{\frac{(4\pi-1)\Lambda_1\mu_2}{32\pi d_{4\pi}}}\min\left\{1,\frac{\la_1+\Lambda_1}{\Lambda_1}\right\},\sqrt{\frac{(4\pi-1)\Lambda_1\mu_1}{32\pi d_{4\pi}}}\min\left\{1,\frac{\la_2+\Lambda_1}{\Lambda_1}\right\}\right\}.
\end{align*}

\end{lemma}

\begin{remark}In particular, in the case of $\Omega=B(0,r)$, it is well known that $\Lambda_1(r)=\Lambda_1(B(0,r))\rg\iy$, as $r\rg0$. In this case for any fixed $\la_1,\la_2,\mu_1,\mu_2$, one has
$$
\beta^\ast>\min\left\{\sqrt{\frac{\mu_1\mu_2}{32}},\sqrt{\frac{\mu_1\mu_2}{32}\frac{\la_1+\Lambda_1}{\Lambda_1}},
\sqrt{\frac{\mu_1\mu_2}{32}\frac{\la_2+\Lambda_1}{\Lambda_1}}\right\}.
$$
\end{remark}

\begin{proof}[Proof of Lemma \ref{lemma:estimatebeta*}]

\noindent We \textbf{claim}:
\begin{equation}\label{est_s4}
\mathcal{S}_4>\sqrt{\frac{(4\pi-1)\Lambda_1}{2d_{4\pi}}},
\end{equation}
where we recall that
$$
d_{4\pi}=\sup_{\stackrel{u\in H^1(\R^2)}{\|\na u\|_2^2+\|u\|_2^2\le1}}\int_{\R^2}\left(e^{4\pi u^2}-1\right)\,\ud x.
$$
We also recall that in \cite{Ruf}, B. Ruf has proved that $d_{4\pi}<\iy$ and that it is attained, see also \cite{ishiwata}.

\noindent Let us prove the claim. In fact, there exists $\vp\in H_0^1(\Omega)$ such that $\|\na\vp\|_2=1$ and $\|\na\vp\|_2^2=\mathcal{S}_4\|\vp\|_4^2$. By Lemma \ref{at} with $\gamma=1$,
$$
\int_{\Omega}\left(e^{\vp^2}-1\right)\,\ud x\le\frac{d_{4\pi}}{4\pi-1}\|\vp\|_2^2.
$$
Noting that $x^2<2(e^x-1)$ for any $x>0$, then
$$
\int_{\Omega}\vp^4\,\ud x<2\int_{\Omega}\left(e^{\vp^2}-1\right)\,\ud x\le\frac{2d_{4\pi}}{4\pi-1}\|\vp\|_2^2.
$$
Since $1=\|\na\vp\|_2^2\ge\Lambda_1\|\vp\|_2^2$, we have $\mathcal{S}_4^{-2}=\|\vp\|_4^4<2d_{4\pi}(4\pi-1)^{-1}\Lambda_1^{-1}$ and the claim is proved.
\br It is a long standing and essentially open problem determining best constants in subcritical Sobolev inequalities, namely for the embedding $H_0^1\hookrightarrow L^p$ as $2<p<2^*$. We address this issue by establishing estimates in the spirit of \eqref{est_s4} to the general $L^p$ case in \cite{DJ18}.
\er
\noindent Recall the definitions of $J_{\lambda_i,\lambda_i}$, $u_i:=u_{\lambda_i,\mu_i}$ and $E_i:=E_{\lambda_i,\mu_i}$ provided in Subsection \ref{subsec:limitproblem}. Next we estimate minimal energies $E_i$ with respect to $\lambda_i$, $\mu_i$ and $\mathcal{S}_4$, namely we \textbf{claim}:
\begin{align}
E_1\ge\min\left\{\frac{\pi}{4},\frac{\pi}{4}\frac{\la_1+\Lambda_1}{\Lambda_1},\frac{e^{-2/3}}{36\mu_1}\left[\min\left\{\frac{1}{2},\frac{\lambda_1+\Lambda_1}{2\Lambda_1}\right\}\mathcal{S}_4\right]^2\right\},\label{eq:E1}\\
E_2\ge\min\left\{\frac{\pi}{4},\frac{\pi}{4}\frac{\la_2+\Lambda_1}{\Lambda_1},\frac{e^{-2/3}}{36\mu_2}\left[\min\left\{\frac{1}{2},\frac{\lambda_2+\Lambda_1}{2\Lambda_1}\right\}\mathcal{S}_4\right]^2\right\}.\label{eq:E2}
\end{align}
Indeed, if $E_1\le\frac{\pi}{4}\min\left\{1,\frac{\lambda_1+\Lambda_1}{\Lambda_1}\right\}$, then by using Lemma \ref{bd1}-(iii), $\lambda>-\Lambda_1$, and Poincar\'e's inequality,
\begin{align*}\lab{wo}
E_1&=J_{\lambda_1,\mu_1}(u_1)-\frac{1}{4}\lan J_{\lambda_1,\mu_1}'(u_1),u_1\ran\nonumber\\
&=\frac{1}{4}\int_{\Omega}\left(|\na u_1|^2+\la_1u_1^2\right)\,\ud x+\frac{\mu_1}{4}\int_{\Omega}\left[u_1^2\left(e^{u_1^2}-1\right)-2\left(e^{u_1^2}-1-u_1^2\right)\right]\,\ud x\\
&> \frac{1}{4}\int_{\Omega}\left(|\na u_1|^2+\la_1u_1^2\right)\,\ud x\geq \frac{1}{4}\min\left\{1,\frac{\lambda_1+\Lambda_1}{\Lambda_1}\right\}\int_{\Omega}|\na u_1|^2\,\ud x\ .\nonumber
\end{align*}

Then $\|\na u_1\|_2^2<\max\{1,\frac{\Lambda_1}{\lambda_1+\Lambda_1}\}4E_1\le\pi$. By Lemma \ref{atn} with $\gamma=\pi$,
$$
\int_{\Omega}u_1^2\left(e^{u_1^2}-1\right)\,\ud x\le\frac{C(\pi)}{\pi}\|u_1\|_4^4 =\frac{320}{9}e^{2/3} \|u_1\|_4^4<36e^{2/3}\|u_1\|_4^4,
$$
and
$$
\frac{\int_{\Omega}|\na u_1|^2\,\ud x}{\left[\int_{\Omega}u_1^2\left(e^{u_1^2}-1\right)\,\ud x\right]^{1/2}}\ge\frac{e^{-1/3}}{6}\frac{\|\na u_1\|_2^2}{\|u_1\|_4^2}\ge\frac{e^{-1/3}}{6}\mathcal{S}_4.
$$
Then
\begin{align*}
&\mu_1\left[\int_{\Omega}u_1^2\left(e^{u_1^2}-1\right)\,\ud x\right]^{1/2}=\frac{\int_{\Omega}\left(|\na u_1|^2+\la_1u_1^2\right)\,\ud x}{\left[\int_{\Omega}u_1^2\left(e^{u_1^2}-1\right)\,\ud x\right]^{1/2}}\\
&\ge\min\left\{1,\frac{\lambda_1+\Lambda_1}{\Lambda_1}\right\}\frac{\int_{\Omega}|\na u_1|^2\,\ud x}{\left[\int_{\Omega}u_1^2\left(e^{u_1^2}-1\right)\,\ud x\right]^{1/2}}\ge\min\left\{1,\frac{\lambda_1+\Lambda_1}{\Lambda_1}\right\}\frac{e^{-1/3}}{6}\mathcal{S}_4,
\end{align*}
that is,
$$
\int_{\Omega}\left(|\na u_1|^2+\la_1u_1^2\right)\,\ud x=\mu_1\int_{\Omega}u_1^2\left(e^{u_1^2}-1\right)\,\ud x\ge\frac{e^{-2/3}}{36\mu_1}\left[\min\left\{1,\frac{\lambda_1+\Lambda_1}{\Lambda_1}\right\}\mathcal{S}_4\right]^2.
$$
Thus,
$$
E_1\ge\frac{1}{4}\int_{\Omega}\left(|\na u_1|^2+\la_1u_1^2\right)\,\ud x\ge\frac{e^{-2/3}}{36\mu_1}\left[\min\left\{\frac{1}{2},\frac{\lambda_1+\Lambda_1}{2\Lambda_1}\right\}\mathcal{S}_4\right]^2.
$$
Similarly, if $E_2\le\frac{\pi}{4}\min\{1,\frac{\lambda_2+\Lambda_1}{\Lambda_1}\}$, we have
$$
E_2\ge\frac{e^{-2/3}}{36\mu_2}\left[\min\left\{\frac{1}{2},\frac{\lambda_2+\Lambda_1}{2\Lambda_1}\right\}\mathcal{S}_4\right]^2.
$$
This proves the claims \eqref{eq:E1}--\eqref{eq:E2}.

\smallbreak

 Let us next estimate $\beta_i$. We \textbf{claim}:
\begin{equation}\label{eq:beta_1}
\beta_1>\min\left\{\sqrt{\frac{\mu_1\mu_2}{32}},\sqrt{\frac{\mu_1\mu_2}{32}\frac{\la_1+\Lambda_1}{\Lambda_1}},\frac{e^{-1/3}}{48}\sqrt{\frac{\Lambda_1\mu_2}{\pi d_{4\pi}}(4\pi-1)}\min\left\{1,\frac{\la_1+\Lambda_1}{\Lambda_1}\right\}\right\},
\end{equation}
and
\begin{equation}\label{eq:beta_2}
\beta_2>\min\left\{\sqrt{\frac{\mu_1\mu_2}{32}},\sqrt{\frac{\mu_1\mu_2}{32}\frac{\la_2+\Lambda_1}{\Lambda_1}},\frac{e^{-1/3}}{48}\sqrt{\frac{\Lambda_1\mu_1}{\pi d_{4\pi}}(4\pi-1)}\min\left\{1,\frac{\la_2+\Lambda_1}{\Lambda_1}\right\}\right\}.
\end{equation}
In fact, we observe that
\[
\begin{cases}
\frac{\mu_1}{4}\int_{\Omega}u_1^2\left(e^{u_1^2}-1\right)\,\ud x<E_1<\frac{\mu_1}{2}\int_{\Omega}u_1^2\left(e^{u_1^2}-1\right)\,\ud x,\\
\frac{\mu_2}{4}\int_{\Omega}u_2^2\left(e^{u_2^2}-1\right)\,\ud x<E_2<\frac{\mu_2}{2}\int_{\Omega}u_2^2\left(e^{u_2^2}-1\right)\,\ud x.
\end{cases}
\]
Then
\begin{align*}
&\int_{\Omega}u_1^2u_2^2\,\ud x\le\left[\int_{\Omega}u_1^4\,\ud x\right]^{1/2}\left[\int_{\Omega}u_2^4\,\ud x\right]^{1/2}\\
&\le\left[\int_{\Omega}u_1^2\left(e^{u_1^2}-1\right)\,\ud x\right]^{1/2}\left[\int_{\Omega}u_2^2\left(e^{u_2^2}-1\right)\,\ud x\right]^{1/2}\\
&\le4\sqrt{\frac{E_2}{E_1\mu_1\mu_2}}E_1<2\mu_1\sqrt{\frac{E_2}{E_1\mu_1\mu_2}}\int_{\Omega}u_1^2\left(e^{u_1^2}-1\right)\,\ud x,
\end{align*}
which implies
$$
\beta_1>\sqrt{\frac{\mu_1\mu_2E_1}{4E_2}}.
$$
Similarly,
$$
\beta_2>\sqrt{\frac{\mu_1\mu_2E_2}{4E_1}}.
$$
By \eqref{est_s4}, \eqref{eq:E1}, \eqref{eq:E2} and $E_i\in (0,2\pi)$, we get the desired estimates \eqref{eq:beta_1}-\eqref{eq:beta_2}.

\smallbreak

 Finally, thanks to $E_1, E_2\in(0,2\pi)$ and again by \eqref{est_s4}, \eqref{eq:E1}, \eqref{eq:E2} we have
$$
\beta_3>\sqrt{\frac{(4\pi-1)\Lambda_1\mu_2}{32\pi d_{4\pi}}}\min\left\{1,\frac{\la_1+\Lambda_1}{\Lambda_1}\right\},\qquad
\beta_4>\sqrt{\frac{(4\pi-1)\Lambda_1\mu_1}{32\pi d_{4\pi}}}\min\left\{1,\frac{\la_2+\Lambda_1}{\Lambda_1}\right\}.
$$
Since $\sqrt{\frac{\mu_1\mu_2}{32}\frac{\la_1+\Lambda_1}{\Lambda_1}},
\sqrt{\frac{\mu_1\mu_2}{32}\frac{\la_2+\Lambda_1}{\Lambda_1}}$ remain bounded as $r\rg0$, the proof is concluded.
\end{proof}

\s{The case $\beta>0$ small (weak cooperation): proof of Theorem \ref{Th1}.}\label{sec3}

\subsection{Nehari Manifold}
\noindent We introduce the following Nehari-type set
$$
\mathcal{M}_{\beta}:=
\left\{
(u,v)\in X, u\not\equiv0, v\not\equiv0
\left|
\begin{array}{l}
\int_{\Omega}(|\na u|^2+\la_1u^2)\,\ud x=\int_{\Omega}uH_u(u,v)\,\ud x,\smallbreak\\
\int_{\Omega}(|\na v|^2+\la_2v^2)\,\ud x=\int_{\Omega}vH_v(u,v)\,\ud x,
\end{array}
\right.
\right\}
$$
and the least energy level
$$
c_\beta:=\inf_{(u,v)\in\mathcal{M}_{\beta}}I(u,v).
$$
\bo
We have $\mathcal{M}_{\beta}\not=\emptyset$ for any $\beta\in\R$.
\eo
\bp
In fact, for any $\vp,\phi\in C_0^\iy(\R^2)\setminus\{0\}$ with disjoint supports, there exist $t_0,s_0>0$ such that $(\sqrt{t_0}\vp,\sqrt{s_0}\phi)\in \mathcal{M}_{\beta}$. To show this, it is enough to prove the existence and uniqueness of a positive solution $(t_0,s_0)$ to the system
\[
\begin{cases}
\int_{\Omega}(|\na \vp|^2+\la_1\vp^2)\,\ud x=\mu_1\int_{\Omega}\vp^2(e^{t\vp^2}-1)\,\ud x,\\
\int_{\Omega}(|\na \phi|^2+\la_2\phi^2)\,\ud x=\mu_2\int_{\Omega}\phi^2(e^{s\phi^2}-1)\,\ud x.
\end{cases}
\]
Let $\g(t)=\mu_1\int_{\Omega}\vp^2(e^{t\vp^2}-1)\,\ud x$. Then $\g(t)\in C([0,\iy))$ and $\g(t)>0$ for all $t>0$. Clearly, $\g(0)=0$ and $\g(t)\ge\mu_1t\int_{\Omega}\vp^4\,\ud x\rg\iy$, as $t\rg\iy$. Since $\la_1>-\Lambda_1$, we have $\int_{\Omega}(|\na \vp|^2+\la_1\vp^2)\,\ud x>0$, which yields the existence of $t_0$ by the mean value theorem. The uniqueness is just a consequence of the monotonicity of $e^x$, which implies the strict monotonicity of $\gamma$. The existence and uniqueness of $s_0$ can be obtained in a similar fashion.
\ep
\bo\lab{Nehari}
If $\beta\in(0,\sqrt{\mu_1\mu_2})$, then $\mathcal{M}_{\beta}$ is a $C^1$-manifold of codimension $2$. Moreover, it is a natural constraint, namely, $(u,v)\in E$ is a critical point of $I$ with nontrivial components if and only if $(u,v)\in\mathcal{M}_{\beta}$ is a critical point of $I\left|_{\mathcal{M}_{\beta}}\right.$.
\eo
\bp
For any $(u,v)\in X$, let
\begin{align*}
&G_1(u,v)=\int_{\Omega}(|\na u|^2+\la_1u^2-uH_u(u,v))\,\ud x,\ G_2(u,v)=\int_{\Omega}(|\na v|^2+\la_2v^2-vH_v(u,v))\,\ud x;
\end{align*}
then one can easily check that $G_1, G_2$ are of $C^1$-class on $X$ and that the following hold
\begin{align*}
&\lan \na G_1(u,v),(\vp,\phi)\ran=\int_{\Omega}\left[ 2(\na u\na \vp+\la_1 u\vp)-\vp H_u(u,v)-u\vp H_{uu}(u,v)-u\phi H_{uv}(u,v)\right]\, \ud x,\\
&\lan \na G_2(u,v),(\vp,\phi)\ran=\int_{\Omega}\left[ 2(\na v\na \phi+\la_2 v\phi)-\phi H_v(u,v)-v\phi H_{vv}(u,v)-v\vp H_{uv}(u,v)\right]\, \ud x,
\end{align*}
where $(\vp,\phi)\in X$. Moreover,
$$
\mathcal{M}_{\beta}:=\{(u,v)\in X, u\not\equiv0, v\not\equiv0: G_1(u,v)=G_2(u,v)=0\}.
$$

{\bf Step 1.} $\mathcal{M}_{\beta}$ is a $C^1$-manifold of codimension $2$. It is enough to prove that $\na G_1(u,v)$ and $\na G_2(u,v)$ are linearly independent for any $(u,v)\in \mathcal{M}_{\beta}$. Indeed, if there exist $\al_1,\al_2\in\R$ and $(u,v)\in \mathcal{M}_{\beta}$ such that $\al_1 \na G_1(u,v)+\al_2 \na G_2(u,v)=0$ in $X^\ast$, then $\lan\al_1 \na G_1(u,v)+\al_2 \na G_2(u,v),(\vp,\phi)\ran=0$ for any $(\phi,\vp)\in \mathcal{M}_{\beta}$. In particular, $\lan\al_1 \na G_1(u,v)+\al_2 \na G_2(u,v),(u,0)\ran=0$ and $\lan\al_1 \na G_1(u,v)+\al_2 \na G_2(u,v),(0,v)\ran=0$. Noting that $G_1(u,v)=G_2(u,v)=0$, we get
\be\lab{fc}
\begin{cases}
\al_1\int_{\Omega}[u^2H_{uu}(u,v)-uH_u(u,v)]\,\ud x+\al_2\int_{\Omega}uvH_{uv}(u,v)\,\ud x=0,\\
\al_1\int_{\Omega}uvH_{uv}(u,v)\,\ud x+\al_2\int_{\Omega}[v^2H_{vv}(u,v)-vH_v(u,v)]\,\ud x=0,
\end{cases}
\ee
Observe that
\begin{align*}
u^2H_{uu}(u,v)-uH_u(u,v)&=2\mu_1u^4e^{u^2}+\beta(u^2G_{uu}(u,v)-uG_u(u,v))\\
&=2\mu_1u^4e^{u^2}+\beta[|uv|^2e^{|uv|}-|uv|(e^{|uv|}-1)],
\end{align*}
\begin{align*}
v^2H_{vv}(u,v)-vH_v(u,v)&=2\mu_2v^4e^{v^2}+\beta(v^2G_{vv}(u,v)-vG_v(u,v))\\
&=2\mu_2v^4e^{v^2}+\beta[|uv|^2e^{|uv|}-|uv|(e^{|uv|}-1)],
\end{align*}
and
$$
uvH_{uv}(u,v)= \beta\left(|uv|^2e^{|uv|}+|uv|(e^{|uv|}-1)\right).
$$
Denote by $J$ the matrix
\begin{align*}
\left[
\begin{array}{cc}
\int_{\Omega}[u^2H_{uu}(u,v)-uH_u(u,v)]\,\ud x & \int_{\Omega}uvH_{uv}(u,v)\,\ud x \\
\int_{\Omega}uvH_{uv}(u,v)\,\ud x & \int_{\Omega}[v^2H_{vv}(u,v)-vH_v(u,v)]\,\ud x
\end{array}
\right]\ ,
\end{align*}
then
\begin{align*}
\mbox{det}(J)&=4\mu_1\mu_2\int_{\Omega}u^4e^{u^2}\,\ud x\int_{\Omega}v^4e^{v^2}\,\ud x-4\beta^2\int_{\Omega}|uv|^2e^{|uv|}\,\ud x\int_{\Omega}|uv|(e^{|uv|}-1)\,\ud x\\
&\,\,\,\,\,\,\,\,+2\beta\int_{\Omega}(\mu_1u^4e^{u^2}+\mu_2v^4e^{v^2})\,\ud x\int_{\Omega}\left[|uv|^2e^{|uv|}-|uv|(e^{|uv|}-1)\right]\,\ud x.
\end{align*}
Noting that $\beta>0$ and $xe^x\ge e^x-1$ for any $x\ge0$, we have
$$
\mbox{det}(J)\ge4\mu_1\mu_2\int_{\Omega}u^4e^{u^2}\,\ud x\int_{\Omega}v^4e^{v^2}\,\ud x-4\beta^2\int_{\Omega}|uv|^2e^{|uv|}\,\ud x\int_{\Omega}|uv|(e^{|uv|}-1)\,\ud x.
$$
By Lemma \ref{bd1}-(i) and again since $xe^x\ge e^x-1$ for $x\ge 0$,
\begin{align*}
&\int_{\Omega}|uv|(e^{|uv|}-1)\,\ud x\le\int_{\Omega}|uv|(e^{u^2}-1)^{1/2}(e^{v^2}-1)^{1/2}\,\ud x\\
&\le\left[\int_{\Omega}u^2(e^{u^2}-1)\,\ud x\right]^{1/2}\left[\int_{\Omega}v^2(e^{v^2}-1)\,\ud x\right]^{1/2}\\
&\le\left[\int_{\Omega}u^4e^{u^2}\,\ud x\right]^{1/2}\left[\int_{\Omega}v^4e^{v^2}\,\ud x\right]^{1/2}.
\end{align*}
At the same time we have
$$
\int_{\Omega}|uv|^2e^{|uv|}\,\ud x\le\int_{\Omega}|uv|^2e^{\frac{u^2+v^2}{2}}\,\ud x\le\left[\int_{\Omega}u^4e^{u^2}\,\ud x\right]^{1/2}\left[\int_{\Omega}v^4e^{v^2}\,\ud x\right]^{1/2}\ ,
$$
so that
$$
\mbox{det}(J)\ge4(\mu_1\mu_2-\beta^2)\int_{\Omega}u^4e^{u^2}\,\ud x\int_{\Omega}v^4e^{v^2}\,\ud x>0,\,\mbox{if}\,\, \beta^2<\mu_1\mu_2.
$$
It follows that system \re{fc} only admits one trivial solution. Namely, $\al_1=\al_2=0$.
\vskip0.1in
{\bf Step 2.} $\mathcal{M}_{\beta}$ is a natural constraint. Suppose $(u,v)\in\mathcal{M}_{\beta}$ is a critical point of $I\left|_{\mathcal{M}_{\beta}}\right.$, then there exist two Lagrange multipliers $\al_3,\al_4\in\R$ such that
\be\lab{lmz}
\na I(u,v)+\al_1 \na G_1(u,v)+\al_2 \na G_2(u,v)=0\,\,\,\mbox{in}\,\,X^\ast.
\ee
Taking the test functions $(u,0)$ and $(0,v)$ respectively in \re{lmz} and noting that $G_1(u,v)=G_2(u,v)=0$, we get $J\cdot(\al_1,\al_2)^T=(0,0)^T$.
By Step 1, $\mbox{det}(J)\not=0$ if $\beta\in(0,\sqrt{\mu_1\mu_2})$. Thus, $\al_1=\al_2=0$. That is $\na I(u,v)=0$ in $X^\ast$.
\ep

\subsection{Estimate of the level $c_\beta$}

We first investigate the relationship between the quantities $\int_{\Omega}u^2\left(e^{u^2}-1\right)\,\ud x$ and $\|\na u\|_2$.
\bl\lab{imbed}
Assume that $\{u_n\}\subset H_0^1(\Omega)\setminus\{0\}$ is such that $\|\na u_n\|_2\rg0$, as $n\rg\iy$. Then
$$
\liminf_{n\rg\iy}\frac{\int_{\Omega}|\na u_n|^2\,\ud x}{\left[\int_{\Omega}u_n^2\left(e^{u_n^2}-1\right)\,\ud x\right]^{\frac12}}\ge\mathcal{S}_4.
$$
\el
\bp
Let $$\Omega_n:=\{x\in\Omega:|u_n(x)|\ge\|\na u_n\|_2^{1/4}\}.$$ Then
\begin{align*}
&\int_{\Omega}u_n^2\left(e^{u_n^2}-1\right)\,\ud x=\int_{\Omega_n}u_n^2\left(e^{u_n^2}-1\right)\,\ud x+\int_{\Omega\setminus\Omega_n}u_n^2\left(e^{u_n^2}-1\right)\,\ud x\\
&\le\|\na u_n\|_2^{-1/2}\int_{\Omega_n}u_n^4\left(e^{u_n^2}-1\right)\,\ud x+\int_{\Omega\setminus\Omega_n}u_n^4\,\ud x \cdot \max_{0<x\le\|\na u_n\|_2^{1/2}}\frac{e^x-1}{x}\\
&\le\|\na u_n\|_2^{-1/2}\int_{\Omega}u_n^4\left(e^{u_n^2}-1\right)\,\ud x+e^{\|\na u_n\|_2^{1/2}}\int_{\Omega}u_n^4\,\ud x.
\end{align*}
By Lemma \ref{at} and using also the fact that $(e^x-1)^2\leq 2^{2x}-1$ for all $x\geq 0$, we have, for sufficient large $n$,
\begin{align*}
\int_{\Omega}u_n^4\left(e^{u_n^2}-1\right)\,\ud x&\le\left[\int_{\Omega}u_n^8\,\ud x\right]^{1/2}\left[\int_{\Omega}\left(e^{u_n^2}-1\right)^2\,\ud x\right]^{1/2}\\
&\le\left[\int_{\Omega}u_n^8\,\ud x\right]^{1/2}\left[\int_{\Omega}\left(e^{2u_n^2}-1\right)\,\ud x\right]^{1/2}\\
&\le\left(\frac{d_{4\pi}}{2\pi-1}\right)^{1/2}\|u_n\|_8^4\|u_n\|_2.
\end{align*}
Since $H_0^1(\Omega)\hookrightarrow L^p(\Omega)$ for any $p\geq 1$, there exists $c>0$ (independent of $n$) such that
$$
\int_{\Omega}u_n^2\left(e^{u_n^2}-1\right)\,\ud x\le \left(c\|\na u_n\|_2^{1/2}+e^{\|\na u_n\|_2^{1/2}}\mathcal{S}_4^{-2}\right)\|\na u_n\|_2^4.
$$
This yields the desired result, since $\|\na u_n\|_2\rg0$, as $n\rg\iy$.
\ep

\bl\lab{lower} The level $c_\beta$ is strictly positive for any $\beta>0$. Moreover, there exists $\rho>0$ such that $\liminf_{\beta\rg0}c_\beta\ge\rho$.
\el
\bp
{\bf Step 1.} Let us first check that $c_\beta\geq 0$. Observe that, for any $(u,v)\in\mathcal{M}_{\beta}$ and $p\ge2$,
\begin{align}\lab{am}
I(u,v)&=I(u,v)-\frac{1}{p}\lan I'(u,v),(u,v)\ran\nonumber\\
&=\frac{p-2}{2p}\int_{\Omega}(|\na u|^2+|\na v|^2+\la_1u^2+\la_2v^2)\,\ud x+K_p(u,v),
\end{align}
where
\begin{align*}
K_p(u,v)=&\frac{\mu_1}{p}\int_{\Omega}\left[u^2(e^{u^2}-1)-\frac{p}{2}(e^{u^2}-1-u^2)\right]\,\ud x\\
&+\frac{2\beta}{p}\int_{\Omega}\left[|uv|(e^{|uv|}-1)-\frac{p}{2}(e^{|uv|}-1-|uv|)\right]\,\ud x\\
&+\frac{\mu_2}{p}\int_{\Omega}\left[v^2(e^{v^2}-1)-\frac{p}{2}(e^{v^2}-1-v^2)\right]\,\ud x.
\end{align*}
By Lemma \ref{bd1}-(iii), $K_p(u,v)>0$ if $p\in[2,4]$. Moreover, since $\la_1,\la_2>-\Lambda_1$, then for $i=1,2$,
$$
\int_{\Omega}(|\na u|^2+\la_iu^2)\,\ud x\ge  \min\left\{ 1, \frac{\la_i+\Lambda_1}{\Lambda_1} \right\}\int_{\Omega}|\na u|^2\,\ud x,\,\,\mbox{for any}\,\,u\in H_0^1(\Omega).
$$
So $I(u,v)>0$ for any $(u,v)\in\mathcal{M}_{\beta}$. It follows that $c_\beta\ge0$.
\vskip0.1in
{\bf Step 2.} We next show that $c_\beta>0$. Indeed if not, assuming $c_\beta=0$ we have there exists $(u_n,v_n)\in\mathcal{M}_{\beta}$ such that $I(u_n,v_n)\rg0$, as $n\rg\iy$. Taking $p=4$ in \re{am}, we have $\|\na u_n\|_2\rg0$ and $\|\na v_n\|_2\rg0$, as $n\rg\iy$. By Lemma \ref{imbed}, for any $\alpha\in (0,S_4)$
\begin{align}\lab{guji1}
\alpha\left[\int_{\Omega}u_n^2\left(e^{u_n^2}-1\right)\,\ud x\right]^{\frac12}\le\int_{\Omega}|\na u_n|^2\,\ud x,\qquad
\alpha\left[\int_{\Omega}v_n^2\left(e^{v_n^2}-1\right)\,\ud x\right]^{\frac12}\le\int_{\Omega}|\na v_n|^2\,\ud x,
\end{align}
if $n$ is large enough. In particular, as $n\rg\iy$ we have
\be\lab{van}
\int_{\Omega}u_n^2\left(e^{u_n^2}-1\right),\int_{\Omega}v_n^2\left(e^{v_n^2}-1\right)\,\ud x\rg0.
\ee
Since $(u_n,v_n)\in\mathcal{M}_{\beta}$, we get
\be\lab{guji2}
\int_{\Omega}(|\na u_n|^2+\la_1u_n^2)\,\ud x=\int_{\Omega}\left[\mu_1u_n^2(e^{u_n^2}-1)+\beta|u_nv_n|(e^{|u_nv_n|}-1)\right]\,\ud x,
\ee
and, since $\la_1>-\Lambda_1$, also that
\be\lab{guji3}
\int_{\Omega}(|\na u_n|^2+\la_1u_n^2)\,\ud x\ge \min\left\{1,\frac{\la_1+\Lambda_1}{\Lambda_1} \right\} \int_{\Omega}|\na u_n|^2\,\ud x.
\ee
Meanwhile, by Lemma \ref{bd1}-(i) and H\"older's inequality,
$$
\int_{\Omega}|u_nv_n|(e^{|u_nv_n|}-1)\,\ud x\le\left[\int_{\Omega}u_n^2\left(e^{u_n^2}-1\right)\,\ud x\right]^{\frac12}\left[\int_{\Omega}v_n^2\left(e^{v_n^2}-1\right)\,\ud x\right]^{\frac12}.
$$
Thus, by \re{guji1}, \re{guji2}, \re{guji3} and since $\beta>0$, for $n$ large enough we have that
\be\lab{guji4}
\alpha \min\left\{ 1,\frac{\la_1+\Lambda_1}{\Lambda_1} \right\}\le\mu_1\left[\int_{\Omega}u_n^2\left(e^{u_n^2}-1\right)\,\ud x\right]^{\frac12}+\beta\left[\int_{\Omega}v_n^2\left(e^{v_n^2}-1\right)\,\ud x\right]^{\frac12}
\ee
which contradicts \re{van}.
\vskip0.1in
{\bf Step 3.} Let us prove that $\liminf_{\beta\to 0^+} c_\beta>0$. If not, there exists $\{\beta_k\}\subset(0,\iy)$ such that $\beta_k\rg0$, as $k\rg\iy$ and $c_{\beta_k}\rg0$. Moreover, there exists $(u_k,v_k)\in\mathcal{M}_{\beta_k}$ such that $I(u_k,v_k)\rg0$, as $k\rg\iy$. Similarly to Step 2, we get a contradiction. This completes the proof of the lemma.
\ep

\noindent As a consequence of Lemma \ref{lower}, we know that $\mathcal{M}_\beta$ is bounded away from the origin.

\bc
Given $\beta>0$, we have
$$
\inf\left\{\|\na u\|_2+\|\na v\|_2: (u,v)\in\mathcal{M}_\beta\right\}>0.
$$
\ec
\noindent Now, in what follows, we establish an upper estimate for $c_\beta$, as long as $\beta$ is small. Recall the definitions of $\beta_1,\beta_2$ from \eqref{eq:beta12}, and recall also that $E_i$ denotes the least energy level of the single equation, see \eqref{eq:les:Ei}.
\bl\lab{upper}
If $0<\beta<\min\{\beta_1,\beta_2\}$, then $c_\beta<E_1+E_2$.
\el
\bp
Let $u_i=u_{\lambda_i,\mu_i}$ be a ground state (positive) solution associated to the level $E_i$, $i=1,2$. For any $t,s\ge0$, let $f(t,s)=I(\sqrt{t}u_1,\sqrt{s}u_2)$, that is,
\begin{align*}
f(t,s)=&\frac{1}{2}\int_{\Omega}\left(t|\na u_1|^2+t\la_1u_1^2+s|\na u_2|^2+s\la_2u_2^2\right)\,\ud x\\
&-\frac{1}{2}\int_{\Omega}\left[\mu_1\left(e^{tu_1^2}-1-tu_1^2\right)+\mu_2\left(e^{su_2^2}-1-su_2^2\right)\right.\\
&\,\,\,\,\,\,\,\,\,\,\,\,\,\,\,\,\,\,\,\,\,\left.+2\beta\left(e^{\sqrt{ts}u_1u_2}-1-\sqrt{ts}u_1u_2\right)\right]\,\ud x.
\end{align*}
Our first aim is to show that $f$ has a global maximum at a pair $(t_0,s_0)\in \R^2$, with $t_0,s_0>0$.

\noindent By combining Lemma \ref{bd1}-(iii) with Lemma \ref{atn} and by taking $t,s>0$ small enough so that $t\|\nabla u_1\|_2^2,\, s\|\nabla u_2\|_2^2<4\pi $, we conclude that, for some $\kappa>0$,
\begin{equation*}
\int_{\Omega}\left(e^{tu_1^2}-1-tu_1^2\right)\,\ud x \le  \frac{1}{2}\int_\Omega t u_1^2 (e^{t u_1^2}-1) \le \frac{C(1)}{2} \| \sqrt{t} u_1\|_4^4 =\kappa t^2 \|u_1\|_4^4
\end{equation*}
Analogously,
\begin{equation*}
\int_{\Omega}\left(e^{tu_2^2}-1-tu_2^2\right)\,\ud x   \le \kappa t^2 \|u_2\|_4^4
\end{equation*}
Combining the last two inequalities with Lemma \ref{bd1}-(ii), and using Cauchy-Schwarz's inequality,
$$
\int_{\Omega}\left(e^{\sqrt{ts}u_1u_2}-1-\sqrt{ts} u_1u_2\right)\,\ud x\le st\kappa\|u_1\|^2_4\|u_2\|^2_4.
$$
Then we have, since $\beta>0$ and $\lambda_1,\lambda_2>-\Lambda_1$,
\begin{align*}
f(t,s)&\ge\frac{t}{2}\left(\|\na u_1\|_2^2+\la_1\|u_1\|_2^2\right)+\frac{s}{2}\left(\|\na u_2\|_2^2+\la_2\|u_2\|_2^2\right)\\
&\,\,\,\,\,\,-\frac{t^2}{2}\kappa \mu_1\|u_1\|_4^4-\frac{s^2}{2}\kappa\mu_2\|u_2\|_4^4-\beta st\kappa\|u_1\|_4^2\|u_2\|_4^2\\
&>0,\,\,\mbox{for $s^2+t^2$ small enough.}
\end{align*}
On the other hand, since $e^x-1-x\ge x^2/2$ for any $x\geq0$,
\begin{align*}
f(t,s)&\le\frac{t}{2}\left(\|\na u_1\|_2^2+\la_1\|u_1\|_2^2\right)+\frac{s}{2}\left(\|\na u_2\|_2^2+\la_2\|u_2\|^2\right)\\
&\,\,\,\,\,\,-\frac{t^2}{4}\mu_1\|u_1\|_4^4-\frac{s^2}{4}\mu_2\|u_2\|_4^4-\frac{\beta}{2} st\|u_1\|_4^2\|u_2\|_4^2\\
&<0,\,\,\mbox{for $s^2+t^2$ large enough.}
\end{align*}
Clearly, $f\in C^1([0,\iy)\times[0,\iy))$. Then there exist $t_0,s_0\ge0$ such that
$$
f(t_0,s_0)=\max_{t,s\ge0}f(s,t)>0.
$$
We claim that $t_0,s_0>0$. If not, without loss of generality we assume that $t_0>0$ and $s_0=0$. So $\frac{\pl f}{\pl t}(t_0,0)=0$ and $\frac{\pl f}{\pl s}(t_0,0)\le0$. That is,
\be\lab{inter1}
\frac{\pl f}{\pl t}(t_0,0)=\frac{1}{2}\int_{\Omega}\left(|\na u_1|^2+\la_1u_1^2\right)\,\ud x-\frac{1}{2}\int_{\Omega}\mu_1u_1^2\left(e^{t_0u_1^2}-1\right)\,\ud x=0,
\ee
and
\be\lab{inter2}
\frac{\pl f}{\pl s}(t_0,0)=\frac{1}{2}\int_{\Omega}\left(|\na u_2|^2+\la_2u_2^2\right)\,\ud x-\frac{t_0}{2}\int_{\Omega}\beta u_1^2u_2^2\,\ud x\le0.
\ee
Recalling that $u_1$ is a positive solution of \re{lm1} for $i=1$, we get\
$$
\int_{\Omega}\left(|\na u_1|^2+\la_1u_1^2\right)\,\ud x=\int_{\Omega}\mu_1u_1^2\left(e^{u_1^2}-1\right)\,\ud x.
$$
It follows from \re{inter1} that $t_0=1$. Then \re{inter2} reduces to
$$
\int_{\Omega}\left(|\na u_2|^2+\la_2u_2^2\right)\,\ud x\le\beta \int_{\Omega}u_1^2u_2^2\,\ud x.
$$
On the other hand, since $u_2$ is a positive solution of \re{lm1} for $i=2$, we get
$$
\int_{\Omega}\mu_2u_2^2\left(e^{u_2^2}-1\right)\,\ud x   = \int_\Omega (|\nabla u_2|^2+\lambda_2 u_2^2)\, \ud x  \le\beta \int_{\Omega}u_1^2u_2^2\,\ud x.
$$
Thus $\beta\geq \beta_2$, which contradict the choice of $\beta$. Observe that if we assumed that $t_0=0$ and $s_0>0$, we would obtain $\beta\geq \beta_1$, again a contradiction.

\noindent Therefore $t_0, s_0>0$ and $\frac{\pl f}{\pl t}(t_0,s_0)=\frac{\pl f}{\pl s}(t_0,s_0)=0$, namely
\[
\langle \nabla I(\sqrt{t_0} u_1,\sqrt{s_0} u_2), ((2\sqrt{t_0})^{-1} u_1,0) \rangle = \langle \nabla I(\sqrt{t_0} u_1,\sqrt{s_0} u_2), (0,(2\sqrt{s_0})^{-1} u_2) \rangle =0,
\]
which implies that $(\sqrt{t_0}u_1,\sqrt{s_0}u_2)\in\mathcal{M}_\beta$. Since $\beta>0$,
\begin{align*}
c_\beta&\le I(\sqrt{t_0}u_1,\sqrt{s_0}u_2)<J_1(\sqrt{t_0}u_1)+J_2(\sqrt{s_0}u_2) -\beta \int_\Omega (e^{\sqrt{t_0s_0}u_1u_2}-1-\sqrt{t_0s_0}u_1u_2)\, \ud x\\
& <J_1(\sqrt{t_0}u_1)+J_2(\sqrt{s_0}u_2) \le\max_{t>0}J_1(\sqrt{t}u_1)+\max_{s>0}J_2(\sqrt{s}u_2)\\
&=J_1(u_1)+J_2(u_2)=E_1+E_2.
\end{align*}
This completes the proof.
\ep

\subsection{Palais-Smale sequence at level $c_\beta$}
By Proposition \ref{Nehari} and Lemma \ref{lower}, we can apply the Ekeland variational principle  \cite{Ekeland}, showing
there exists a minimization sequence $\{(u_n,v_n)\}\subset \mathcal{M}_\beta$ such that
\be\lab{mini}
I(u_n,v_n)\rg c_\beta,\,\,\na I|_{\mathcal{M}_\beta}(u_n,v_n)\rg0\,\,\mbox{in}\,\, X^\ast, \text{ as } n\rg\iy\ .
\ee

\noindent Clearly $\{(u_n,v_n)\}$ depends on $\beta$ and, for any $\beta>0$ fixed, the minimization sequence $\{(u_n,v_n)\}$ may not be unique. Recall from \eqref{eq:defSlambdamu} that, for $\lambda>-\Lambda_1$, $\mu>0$, $S_{\lambda,\mu}$ denotes the set of ground states of the single equation \eqref{limit}. Given $\delta>0$, denote by $(S_{\lambda,\mu})^\delta$ the neighborhood of $S_{\lambda,\mu}$ of radius $\delta$. We have the following.
\bl \label{neighborhood}
For any $\delta>0$, there exists $\beta_\delta>0$ such that for any $\beta\in(0,\beta_\delta)$, up to a subsequence, there exists $\{(u_n^\beta,v_n^\beta)\}\subset \mathcal{M}_\beta$ satisfying \re{mini} and $\{u_n^\beta\}\subset(S_{\la_1,\mu_1})^\delta$ and $\{v_n^\beta\}\subset(S_{\la_2,\mu_2})^\delta$.
\el
\bp
Suppose by contradiction 	the lemma does not hold. Then for some $\delta_0>0$, there exists $\{\beta_k\}\subset\R^+$ such that, $\beta_k\rg0$, as $k\rg\iy$, and for any $\{(u_n^{\beta_k},v_n^{\beta_k})\}\subset \mathcal{M}_{\beta_k}$ satisfying \re{mini}, there holds $\{u_n^{\beta_k}\}\subset H_0^1(\Omega)\backslash(S_{\la_1,\mu_1})^{\delta_0}$ and $\{v_n^{\beta_k}\}\subset H_0^1(\Omega)\backslash(S_{\la_2,\mu_2})^{\delta_0}$. For any $k$, there exists $n_k$ such that
$$
\left|I(u_{n_k}^{\beta_k},v_{n_k}^{\beta_k})-c_{\beta_k}\right|\le1/k.
$$
Let $\ti{u}_k=u_{n_k}^{\beta_k}$ and $\ti{v}_k=v_{n_k}^{\beta_k}$, then
$$
\limsup_{k\rg\iy}I(\ti{u}_k,\ti{v}_k)=\limsup_{k\rg\iy}c_{\beta_k}\le E_{\la_1,\mu_1}+E_{\la_2,\mu_2}.
$$
By \re{am} there exists $C>0$ such that for all $k$,
\be\lab{bound1}
\|\na \ti{u}_k\|_2, \|\na \ti{v}_k\|_2, \int_{\Omega}\ti{u}_k^2\left(e^{\ti{u}_k^2}-1\right)\,\ud x,\int_{\Omega}\ti{v}_k^2\left(e^{\ti{v}_k^2}-1\right)\,\ud x\le C.
\ee
Up to a subsequence, we may assume that $\ti{u}_k\rg u$ and $\ti{v}_k\rg v$ weakly in $H_0^1(\Omega)$ and a. e. in $\Omega$, as $k\rg\iy$.
By Lemma \ref{lower}, we have
\be\lab{ne1}
\liminf_{k\rg\iy}\min\{\|\na \ti{u}_k\|_2, \|\na \ti{v}_k\|_2\}>0.
\ee
Noting that $\beta_k>0$ and
\be\lab{ne2}
\int_{\Omega}(|\na\ti{u}_k|^2+\la_1\ti{u}_k^2)\,\ud x=\mu_1\int_{\Omega}\ti{u}_k^2(e^{\ti{u}_k^2}-1)\,\ud x+o_k(1),
\ee
there exists $t_k\in[1,\iy)$ such that
\be\lab{ne3}
\int_{\Omega}(|\na\ti{u}_k|^2+\la_1\ti{u}_k^2)\,\ud x=\mu_1\int_{\Omega}\ti{u}_k^2(e^{t_k\ti{u}_k^2}-1)\,\ud x,
\ee
that is $\sqrt{t_k}\ti{u}_k\in\mathcal{N}_{\la_1,\mu_1}$. Similarly, there exists $s_k\in[1,\iy)$ such that $\sqrt{s_k}\ti{v}_k\in\mathcal{N}_{\la_2,\mu_2}$.
\vskip0.1in
{\bf Step 1.} We claim that $t_k\rg1$ and $s_k\rg1$, as $k\rg\iy$. We only give the proof of $t_k\rg1$, as the second convergence being similar. We consider two cases:
\vskip0.1in
{\bf Case I.} $u\not=0$. If $\limsup_{k\rg\iy}t_k>1$, then we can assume that $t_k>1$ for all $k$. By \re{ne2} and \re{ne3} we have
$$
(t_k-1)\int_{\Omega}\ti{u}_k^4\,\ud x\le\int_{\Omega}\ti{u}_k^2(e^{t_k\ti{u}_k^2}-e^{\ti{u}_k^2})\,\ud x=o_k(1),
$$
which yields $t_k\rg1$ as $k\rg\iy$. This is a contradiction. So $\limsup_{k\rg\iy}t_k\le1$. Similarly, $\liminf_{k\rg\iy}t_k\ge1$. Then $\lim_{k\rg\iy}t_k=1$.

{\bf Case II.} $u=0$. If $\limsup_{k\rg\iy}t_k>1$, then we can assume that $t_k>t_0>1$ for all $k$. Noting that for any $\e>0$, there exists $R_\e>0$ such that
\begin{align*}
&\limsup_{k\rg\iy}\int_{\Omega}\ti{u}_k^2(e^{\ti{u}_k^2}-1)\,\ud x\\
&\le\e\limsup_{k\rg\iy}\int_{\{|\ti{u}_k|\ge R_\e\}}\ti{u}_k^2(e^{t_k\ti{u}_k^2}-1)\,\ud x+\limsup_{k\rg\iy}\int_{\{|\ti{u}_k|\le R_\e\}}\ti{u}_k^2(e^{\ti{u}_k^2}-1)\,\ud x\\
&\le C\e.
\end{align*}
Since $\e$ is arbitrary, we have $\limsup_{k\rg\iy}\int_{\Omega}\ti{u}_k^2(e^{\ti{u}_k^2}-1)\,\ud x=0$, which contradicts \re{ne1} and \re{ne2}. So $\limsup_{k\rg\iy}t_k\le1$. Similarly, $\liminf_{k\rg\iy}t_k\ge1$. Thus, $\lim_{k\rg\iy}t_k=1$.
\vskip0.1in
{\bf Step 2.} Let $\bar{u}_k=\sqrt{t_k}\ti{u}_k$ and $\bar{v}_k=\sqrt{s_k}\ti{v}_k$, then $\bar{u}_k\rg u$ and $\bar{v}_k\rg v$ weakly in $H_0^1(\Omega)$ and a.e. in $\Omega$, as $k\rg\iy$. In the following, we adopt some idea in \cite{Ad} to show that $u\in S_{\la_1,\mu_1}$, $v\in S_{\la_2,\mu_2}$ and $\bar{u}_k\rg u$, $\bar{v}_k\rg v$ strongly in $H_0^1(\Omega)$, as $k\rg\iy$. This will be a contradiction.

\noindent By Step 1, we know that $\|\na(\bar{u}_k-\ti{u}_k)\|_2\rg0$ and $\|\na(\bar{v}_k-\ti{v}_k)\|_2\rg0$ as $k\rg\iy$. So
\begin{align*}
I(\ti{u}_k,\ti{v}_k)&=I(\bar{u}_k,\bar{v}_k)+o_k(1)=J_{\la_1,\mu_1}(\bar{u}_k)+J_{\la_2,\mu_2}(\bar{v}_k)+o_k(1)\\
&\ge E_{\la_1,\mu_1}+E_{\la_2,\mu_2}+o_k(1).
\end{align*}
Recalling that $\limsup_{k\rg\iy}I(\ti{u}_k,\ti{v}_k)\le E_{\la_1,\mu_1}+E_{\la_2,\mu_2}$, we get
$$
\lim_{k\rg\iy}J_{\la_1,\mu_1}(\bar{u}_k)=E_{\la_1,\mu_1},\,\,\lim_{k\rg\iy}J_{\la_2,\mu_2}(\bar{u}_k)=E_{\la_2,\mu_2}.
$$

\noindent Now, arguing as in the proof of Lemma \ref{priori}, we deduce that $u\not\equiv0$ and
\be\lab{upp}
\int_{\Omega}\left(|\na u|^2+\la_1u^2\right)\,\ud x=\int_{\Omega}\mu_1u^2\left(e^{u^2}-1\right)\,\ud x.
\ee

\noindent By \re{upp}, there exists $t^\ast\in(0,1]$ such that $\sqrt{t^\ast} u\in \mathcal{N}_{\la_1,\mu_1}$ and then
\begin{align*}
&E_{\la_1,\mu_1}\le J_{\la_1,\mu_1}(\sqrt{t^\ast} u)\\
&=\frac{\mu_1}{2}\int_{\Omega}\left[t^\ast u^2\left(e^{t^\ast u^2}-1\right)-\left(e^{t^\ast u^2}-1-t^\ast u^2\right)\right]\,\ud x\\
&\le\frac{\mu_1}{2}\liminf_{k\rg\iy}\int_{\Omega}\left[t^\ast \bar{u}_k^2\left(e^{t^\ast \bar{u}_k^2}-1\right)-\left(e^{t^\ast \bar{u}_k^2}-1-t^\ast \bar{u}_k^2\right)\right]\,\ud x\\
&\le\frac{\mu_1}{2}\liminf_{k\rg\iy}\int_{\Omega}\left[\bar{u}_k^2\left(e^{\bar{u}_k^2}-1\right)-\left(e^{\bar{u}_k^2}-1-\bar{u}_k^2\right)\right]\,\ud x\\
&=\liminf_{k\rg\iy}J_{\la_1,\mu_1}(\bar{u}_k)=E_{\la_1,\mu_1},
\end{align*}
where we used the fact that the function $x(e^x-1)-(e^x-1-x)$ is strictly increasing in $[0,\iy)$. Thus $t^\ast=1$ and $J_{\la_1,\mu_1}(u)=E_{\la_1,\mu_1}$, which is a contradiction.

\noindent Finally, we can similarly prove that $\bar{v}_k\rg v$ strongly in $H_0^1(\Omega)$ and so $v\in S_{\la_2,\mu_2}$. By Step 1, we know that $\ti{u}_k\rg u$ and $\ti{v}_k\rg v$ strongly in $H_0^1(\Omega)$, as $k\rg\iy$, which contradicts that the fact that $\ti{v}_k\in  H_0^1(\Omega)\backslash(S_{\la_1,\mu_1})^{\delta_0}$ and $\ti{v}_k\in H_0^1(\Omega)\backslash(S_{\la_2,\mu_2})^{\delta_0}$ for any $k$. This completes the proof.
\ep
\bc\lab{tl}
There exists $\beta^{\ast\ast}>0$ such that for any fixed $\beta\in(0,\beta^{\ast\ast})$, up to a subsequence, there exists $\{(u_n,v_n)\}\subset \mathcal{M}_\beta$ satisfying \re{mini} and
$$
\sup_n\int_{\Omega}u_n^4e^{u_n^2}\,\ud x<\iy,\,\,\sup_n\int_{\Omega}v_n^4e^{v_n^2}\,\ud x<\iy.
$$
\ec
\bp
By Lemma \ref{neighborhood}, similarly to Lemma \ref{priori}, one can show that, there exists some $\beta^{\ast\ast}>0$ such that for any fixed $\beta\in(0,\beta^{\ast\ast})$, up to a subsequence, there exists $\{(u_n,v_n)\}\subset \mathcal{M}_\beta$ satisfying \re{mini} and $\{(u_n,v_n)\}$ is bounded in $L^\iy(\Omega)\times L^\iy(\Omega)$. Noting that $\sup_n\int_{\Omega}u_n^2e^{u_n^2}\,\ud x<\iy,\,\,\sup_n\int_{\Omega}v_n^2e^{v_n^2}\,\ud x<\iy$, the proof is concluded.
\ep
\noindent Now let $\{(u_n,v_n)\}\subset \mathcal{M}_\beta$ be as in Corollary \ref{tl}. Then by \re{am} there exists $C>0$ such that for all $n$,
\be\lab{bound}
\|\na u_n\|_2, \|\na v_n\|_2, \int_{\Omega}u_n^2\left(e^{u_n^2}-1\right)\,\ud x,\int_{\Omega}v_n^2\left(e^{v_n^2}-1\right)\,\ud x\le C.
\ee
Without loss of generality, we may assume that $u_n\rg u$, $v_n\rg v$ weakly in $H_0^1(\Omega)$, strongly in $L^2(\Omega)$ and a.e. in $\Omega$, as $n\rg\iy$. As a consequence of \cite[Lemma 2.1]{FMR}, we get
\be\lab{con}
\left\{
\begin{array}{ll}
\lim_{n\rg\iy}\int_{\Omega}\left(e^{u_n^2}-1-u_n^2\right)\,\ud x=\int_{\Omega}\left(e^{u^2}-1-u^2\right)\,\ud x,\\
\lim_{n\rg\iy}\int_{\Omega}\left(e^{v_n^2}-1-v_n^2\right)\,\ud x=\int_{\Omega}\left(e^{v^2}-1-v^2\right)\,\ud x,
\end{array}
\right.
\ee
Furthermore, we have
\bo\lab{Palais-Smale} Let $0<\beta<\beta_0=\min\{\sqrt{\mu_1\mu_2},\beta^\ast,\beta^{\ast\ast}\}$. Then $(u,v)$ is a critical point of $I$.
\eo
\bp
\vskip0.1in
{\bf Step 1.} {\bf Claim}:
\be\lab{gradient}
\liminf_{n\rg\iy}\|\na u_n\|_2>0\,\,\mbox{and}\,\,\liminf_{n\rg\iy}\|\na v_n\|_2>0.
\ee
\noindent By contradiction, if the claim does not hold, without loss of generality, we may assume $\lim\limits_{n\rg\iy}\|\na u_n\|_2=0$ and so $u\equiv 0$. Then, similarly to Lemma \ref{lower}, for any $\al\in(0,\mathcal{S}_4)$ and $n$ large enough,
$$
\al\min\left\{1,\frac{\la_1+\Lambda_1}{\Lambda_1}\right\}\le\mu_1\left[\int_{\Omega}u_n^2\left(e^{u_n^2}-1\right)\,\ud x\right]^{\frac12}+\beta\left[\int_{\Omega}v_n^2\left(e^{v_n^2}-1\right)\,\ud x\right]^{\frac12},
$$
and
\begin{equation}\label{conv_to0}
\lim_{n\rg\iy}\int_{\Omega}u_n^2\left(e^{u_n^2}-1\right)\,\ud x=0.
\end{equation}
It follows that
\be\lab{low}
\liminf_{n\rg\iy}\int_{\Omega}v_n^2\left(e^{v_n^2}-1\right)\,\ud x\ge\left[\frac{\mathcal{S}_4}{\beta}\min\left\{1,\frac{\la_1+\Lambda_1}{\Lambda_1}\right\}\right]^2.
\ee
Moreover, it follows from Lemma \ref{bd1}, \eqref{bound}, \Hugo{\eqref{con} and \eqref{conv_to0}},
$$
\lim_{n\rg\iy}\int_{\Omega}|u_nv_n|\left(e^{|u_nv_n|}-1\right)\,\ud x=\lim_{n\rg\iy}\int_{\Omega}\left(e^{|u_nv_n|}-1-|u_nv_n|\right)\,\ud x=0.
$$
So
\begin{align*}
I(u_n,v_n)=\frac{1}{2}\int_{\Omega}\left(|\na v_n|^2+\la_2|v_n|^2\right)\,\ud x-\frac{\mu_2}{2}\int_{\Omega}\left(e^{v_n^2}-1-v_n^2\right)\,\ud x+o_n(1),
\end{align*}
where $o_n(1)\rg0$ as $n\rg\iy$. Noting that $(u_n,v_n)\in\mathcal{M}_\beta$,
$$
\int_{\Omega}\left(|\na v_n|^2+\la_2|v_n|^2\right)\,\ud x=\mu_2\int_{\Omega}v_n^2\left(e^{v_n^2}-1\right)\,\ud x+o_n(1),
$$
which implies that
\begin{align*}
I(u_n,v_n)&=\frac{\mu_2}{2}\int_{\Omega}\left[v_n^2\left(e^{v_n^2}-1\right)-\left(e^{v_n^2}-1-v_n^2\right)\right]\,\ud x+o_n(1)\\
&\ge\frac{\mu_2}{4}\int_{\Omega}v_n^2\left(e^{v_n^2}-1\right)\,\ud x+o_n(1),
\end{align*}
where we used Lemma \ref{bd1}. Then by \re{low},
$$
c_\beta=\liminf_{n\rg\iy}I(u_n,v_n)\ge\frac{\mu_2}{4}\left[\frac{\mathcal{S}_4}{\beta}\min\left\{1,\frac{\la_1+\Lambda_1}{\Lambda_1}\right\}\right]^2.
$$
By Lemma \ref{upper} and the choice of $\beta<\beta_3$, we get a contradiction. On the other hand, if we assume that $\lim\limits_{n\rg\iy}\|\na v_n\|_2=0$ we get a contradiction from $\beta<\beta_4$.Thus, the claim is true.

\vskip0.1in
{\bf Step 2.} By the Ekeland variational principle \cite{Ekeland}, there exists $\{(t_n,s_n)\}\subset\R^2$ such that
\be\lab{ps}
\na I|_{\mathcal{M}_\beta}(u_n,v_n)=\na I(u_n,v_n)-t_n\na G_1(u_n,v_n)-s_n\na G_2(u_n,v_n)\rg0\,\,\mbox{in}\,\, X^\ast.
\ee
We claim that $t_n, s_n\rg0$ as $n\rg\iy$. By \re{ps} and $(u_n,v_n)\in\mathcal{M}_\beta$, we have
\be\lab{sys}
\begin{cases}
\lan t_n\na G_1(u_n,v_n)+s_n\na G_2(u_n,v_n),(u_n,0)\ran=o_n(1),\\
\lan t_n\na G_1(u_n,v_n)+s_n\na G_2(u_n,v_n),(0,v_n)\ran=o_n(1).
\end{cases}
\ee
Similar to Proposition \ref{Nehari}, system \re{sys} is equivalent to
\be\lab{matrix}
J_n\cdot(t_n,s_n)^T=(o_n(1),o_n(1))^T,
\ee
where
\begin{align*}
J_n=\left(
\begin{array}{cc}
a_n & c_n \\
c_n & b_n
\end{array}
\right),
\end{align*}
\[
\begin{cases}
a_n=\int_{\Omega}\left[2\mu_1u_n^4e^{u_n^2}+\beta[|u_nv_n|^2e^{|u_nv_n|}-|u_nv_n|(e^{|u_nv_n|}-1)]\right]\,\ud x,\\
b_n=\int_{\Omega}\left[2\mu_2v_n^4e^{v_n^2}+\beta[|u_nv_n|^2e^{|u_nv_n|}-|u_nv_n|(e^{|u_nv_n|}-1)]\right]\,\ud x,\\
c_n=\int_{\Omega}\left[|u_nv_n|^2e^{|u_nv_n|}+|u_nv_n|(e^{|u_nv_n|}-1)\right]\,\ud x,
\end{cases}
\]
and
\be\lab{deter}
\mbox{det}(J_n)\ge4(\mu_1\mu_2-\beta^2)\int_{\Omega}u_n^4e^{u_n^2}\,\ud x\int_{\Omega}v_n^4e^{v_n^2}\,\ud x>0.
\ee
By Step 1, there exists $c>0$ such that for all $n$,
\be\lab{key}
\int_{\Omega}u_n^4e^{u_n^2}\,\ud x,\int_{\Omega}v_n^4e^{v_n^2}\,\ud x\ge c.
\ee
If not, we assume that
$$
\lim_{n\rg\iy}\int_{\Omega}u_n^4e^{u_n^2}\,\ud x=0,
$$
then
$$
\lim_{n\rg\iy}\int_{\Omega}u_n^2\left(e^{u_n^2}-1\right)\,\ud x\le\lim_{n\rg\iy}\int_{\Omega}u_n^4e^{u_n^2}\,\ud x=0.
$$
Then, similar as above, we know $\|\na u_n\|_2\rg0$, as $n\rg\iy$, which is a contradiction.

\noindent By Cramer's rule,
$$
t_n=o_n(1)\frac{b_n}{\mbox{det}(J_n)}+o_n(1)\frac{c_n}{\mbox{det}(J_n)},\,\,s_n=o_n(1)\frac{a_n}{\mbox{det}(J_n)}+o_n(1)\frac{c_n}{\mbox{det}(J_n)}.
$$
Now, since
\begin{align*}
0<a_n\le 2\mu_1\int_{\Omega}u_n^4e^{u_n^2}\,\ud x+\beta\left[\int_{\Omega}u_n^4e^{u_n^2}\,\ud x\right]^{1/2}\left[\int_{\Omega}v_n^4e^{v_n^2}\,\ud x\right]^{1/2},
\end{align*}
by \re{deter}, there exists $C_1>0$ (independent of $n$) such that
\be\lab{van1}
0<\frac{a_n}{\mbox{det}(J_n)}\le C_1\left[\left(\int_{\Omega}u_n^4e^{u_n^2}\,\ud x\right)^{-1}+\left(\int_{\Omega}v_n^4e^{v_n^2}\,\ud x\right)^{-1}\right].
\ee
Similarly,
\be\lab{van2}
0<\frac{b_n}{\mbox{det}(J_n)},\frac{c_n}{\mbox{det}(J_n)}\le C_1\left[\left(\int_{\Omega}u_n^4e^{u_n^2}\,\ud x\right)^{-1}+\left(\int_{\Omega}v_n^4e^{v_n^2}\,\ud x\right)^{-1}\right].
\ee
It follows from \re{key} that $\left\{\frac{a_n}{\mbox{det}(J_n)}\right\}$, $\left\{\frac{b_n}{\mbox{det}(J_n)}\right\}$ and $\left\{\frac{c_n}{\mbox{det}(J_n)}\right\}$ are bounded. Thus $t_n,s_n\rg0$, as $n\rg\iy$.

\vskip0.1in
{\bf Step 3.} We claim that for any fixed $\vp,\phi\in C_0^\iy(\Omega)$, $\lan I'(u_n,v_n),(\vp,\phi)\ran\rg0$, as $n\rg\iy$, which implies that $I'(u,v)=0$ in $X^\ast$. By \re{ps},
$$
\lan I'(u_n,v_n),(\vp,\phi)\ran=t_n\lan\na G_1(u_n,v_n),(\vp,\phi)\ran+s_n\lan\na G_2(u_n,v_n),(\vp,\phi)\ran+o_n(1).
$$
On the one hand,
\begin{align*}
\lan I'(u_n,v_n),(\vp,\phi)\ran=&\int_{\Omega}(\na u_n\na\vp+\na v_n\na\phi+\la_1u_n\vp+\la_2v_n\phi)\\
&-\int_{\Omega}H_u(u_n,v_n)\vp+H_v(u_n,v_n)\phi\\
=&\int_{\Omega}(\na u\na\vp+\na v\na\phi+\la_1u\vp+\la_2v\phi)\\
&-\int_{\Omega}H_u(u_n,v_n)\vp+H_v(u_n,v_n)\phi+o_n(1).
\end{align*}
Since
$$
\int_{\Omega}u_nH_u(u_n,v_n)\,\ud x=\int_{\Omega}\left(\mu_1u_n^2(e^{u_n^2}-1)+\beta|u_nv_n|(e^{|u_nv_n|}-1)\right)\,\ud x
$$
is bounded uniformly for $n$, by \cite[Lemma 2.1]{FMR} we get that
$$
\lim_{n\rg\iy}\int_{\Omega}H_u(u_n,v_n)\vp=\int_{\Omega}H_u(u,v)\vp.
$$
Similarly,
$$
\lim_{n\rg\iy}\int_{\Omega}H_v(u_n,v_n)\phi=\int_{\Omega}H_v(u,v)\phi.
$$
Then
$$
\lim_{n\rg\iy}\lan I'(u_n,v_n),(\vp,\phi)\ran=\lan I'(u,v),(\vp,\phi)\ran.
$$
On the other hand, we have that
$$
t_n\lan\na G_1(u_n,v_n),(\vp,\phi)\ran=o_n(1), s_n\lan\na G_2(u_n,v_n),(\vp,\phi)\ran=o_n(1).
$$
We only prove
$$
t_n\lan\na G_1(u_n,v_n),(\vp,\phi)\ran=o_n(1).
$$
The left one can be proved similarly. Notice that
\begin{align*}
\lan \na G_1(u_n,v_n),(\vp,\phi)\ran=&\int_{\Omega}\left[ 2(\na u_n\na \vp+\la_1 u_n\vp)-\vp H_u(u_n,v_n)\right.\\
&\left.-u_n\vp H_{uu}(u_n,v_n)-u_n\phi H_{uv}(u_n,v_n)\right]\, \ud x,
\end{align*}
and
$$
\sup_{n}\left|\int_{\Omega}(\na u_n\na \vp+\la_1 u_n\vp)\, \ud x\right|\le\sup_{n}\left(\|\na u_n\|_2\|\na\vp\|_2+|\la_1|\|u_n\|_2\|\vp\|_2\right)<\iy.
$$
Moreover,
\begin{align*}
&\left|\int_{\Omega}\vp H_u(u_n,v_n)\, \ud x\right|\le\|\vp\|_\iy\int_{\Omega}\left|H_u(u_n,v_n)\right|\, \ud x\\
&\le\|\vp\|_\iy\int_{\Omega}\left[\mu_1|u_n|(e^{u_n^2}-1)+\beta|v_n|(e^{|u_nv_n|}-1)\right]\, \ud x\\
&\le\|\vp\|_\iy\int_{\{|u_n(x)\ge1|\}}\left[\mu_1u_n^2(e^{u_n^2}-1)+\beta|u_nv_n|(e^{|u_nv_n|}-1)\right]\, \ud x\\
&\,\,\,\,\,\,+\|\vp\|_\iy\int_{\{|u_n(x)\le1|\}}\left[\mu_1(e-1)+\beta|v_n|(e^{|v_n|}-1)\right]\, \ud x\\
\end{align*}
Since
$$
\int_{\{|u_n(x)\le1|\}}|v_n|(e^{|v_n|}-1)\, \ud x\le\sqrt{|\Omega|(e-1)}\left[\int_{\Omega}v_n^2(e^{v_n^2}-1)\, \ud x\right]^{1/2}
$$
we have
$$
\sup_{n}\left|\int_{\Omega}\vp H_u(u_n,v_n)\, \ud x\right|<\iy.
$$
Then by Step 2, it is enough to show
$$
t_n\int_{\Omega}\left[u_n\vp H_{uu}(u_n,v_n)+u_n\phi H_{uv}(u_n,v_n)\right]\, \ud x=o_n(1).
$$
From
$$
\int_{\Omega}|u_n|^3e^{u_n^2}\, \ud x\le\left(\int_{\Omega}u_n^2e^{u_n^2}\, \ud x\right)^{1/2}\left(\int_{\Omega}u_n^4e^{u_n^2}\, \ud x\right)^{1/2},
$$
$$
\int_{\Omega}|u_n|(e^{u_n^2}-1)\, \ud x\le\left(\int_{\Omega}u_n^2e^{u_n^2}\, \ud x\right)^{1/2}\left(\int_{\Omega}e^{u_n^2}\, \ud x\right)^{1/2},
$$
$$
\int_{\Omega}|u_n|v_n^2e^{|u_nv_n|}\, \ud x\le\left(\int_{\Omega}u_n^2e^{u_n^2}\, \ud x\right)^{1/2}\left(\int_{\Omega}v_n^4e^{v_n^2}\, \ud x\right)^{1/2},
$$
there exists $C_2>0$(independent of $n$) such that
\begin{align}\lab{key2}
&\left|\int_{\Omega}u_n\vp H_{uu}(u_n,v_n)\, \ud x\right|\le\|\vp\|_\iy\int_{\Omega}\left|u_nH_{uu}(u_n,v_n)\right|\, \ud x\nonumber\\
&\le\|\vp\|_\iy\int_{\Omega}\left[2\mu_1|u_n|^3e^{u_n^2}+\mu_1|u_n|(e^{u_n^2}-1)+\beta|u_n|v_n^2e^{|u_nv_n|}\right]\, \ud x\nonumber\\
&\le C_2\left[\left(\int_{\Omega}u_n^4e^{u_n^2}\, \ud x\right)^{1/2}+\left(\int_{\Omega}v_n^4e^{v_n^2}\, \ud x\right)^{1/2}\right]+C_2.
\end{align}

\noindent On the other side, since
$$
\int_{\Omega}u_n^2|v_n|e^{|u_nv_n|}\, \ud x\le\left(\int_{\Omega}v_n^2e^{v_n^2}\, \ud x\right)^{1/2}\left(\int_{\Omega}u_n^4e^{u_n^2}\, \ud x\right)^{1/2},
$$
$$
\int_{\Omega}|u_n|(e^{|u_nv_n|}-1)\, \ud x\le\left(\int_{\Omega}u_n^2e^{u_n^2}\, \ud x\right)^{1/2}\left(\int_{\Omega}e^{v_n^2}\, \ud x\right)^{1/2},
$$
there exists $C_3>0$, which does not depend on $n$, such that
\begin{align}\lab{key3}
&\left|\int_{\Omega}u_n\phi H_{uv}(u_n,v_n)\, \ud x\right|\le\|\phi\|_\iy\int_{\Omega}\left|u_nH_{uv}(u_n,v_n)\right|\, \ud x\nonumber\\
&\le\|\phi\|_\iy\int_{\Omega}\left[u_n^2|v_n|e^{|u_nv_n|}+|u_n|(e^{|u_nv_n|}-1)\right]\, \ud x\nonumber\\
&\le C_3\left(\int_{\Omega}u_n^4e^{u_n^2}\, \ud x\right)^{1/2}+C_3.
\end{align}
We can conclude the proof by combining \re{key2}, \re{key3} and Corollary \ref{tl}.
\ep

\subsection{Conclusion of the proof of Theorem \ref{Th1}}

\bp
Let $(u,v)$ be given in Proposition \ref{Palais-Smale}, then $I'(u,v)=0$. If $u\not\equiv0, v\not\equiv0$, then $(u,v)\in\mathcal{M}_\beta$ and thus
$I(u,v)\ge c_\beta$. By $(u_n,v_n)\in\mathcal{M}_\beta$ and taking $p=4$ in \re{am}, Fatou's lemma yields
$$
c_\beta=\liminf_{n\rg\iy}I(u_n,v_n)\ge I(u,v).
$$
Hence $I(u,v)=c_\beta$. Noting that $(|u|,|v|)\in\mathcal{M}_\beta$ and $I(u,v)=I(|u|,|v|)$, without loss of generality, we may assume $(u,v)$ is a minimizer for $c_\beta$ on $\mathcal{M}_\beta$. By the Lagrange multiplier theorem, there exist $\al,\beta\in\R$ such that
$$
\na I|_{\mathcal{M}_\beta}(u,v)=\na I(u,v)-\al\na G_1(u,v)-\beta\na G_2(u,v)=0\,\,\mbox{in}\,\, X^\ast.
$$
It follows from Proposition \ref{Nehari} that $\al=\beta=0$ and  by the maximum principle $(u,v)$ is a positive ground state solution of \re{q1}. The last statement of Theorem \ref{Th1} is a direct consequence of Lemma \ref{neighborhood} combined with Lemma \ref{priori}.
\vskip0.1in
\noindent Now, we next show that actually $u\not\equiv0, v\not\equiv0$. We proceed by a contradiction argument.
\vskip0.1in
{\bf Case 1.} $u\equiv v\equiv 0$. For any $p>2$ in \re{am},
$$
I(u_n,v_n)=\frac{p-2}{2p}\int_{\Omega}(|\na u_n|^2+|\na v_n|^2+\la_1u_n^2+\la_2v_n^2)\,\ud x+K_p(u_n,v_n),
$$
where
\begin{align*}
K_p(u_n,v_n)=&\frac{\mu_1}{p}\int_{\Omega}\left[u_n^2(e^{u_n^2}-1)-\frac{p}{2}(e^{u_n^2}-1-u_n^2)\right]\,\ud x\\
&+\frac{2\beta}{p}\int_{\Omega}\left[|u_nv_n|(e^{|u_nv_n|}-1)-\frac{p}{2}(e^{|u_nv_n|}-1-|u_nv_n|^2)\right]\,\ud x\\
&+\frac{\mu_2}{p}\int_{\Omega}\left[v_n^2(e^{v_n^2}-1)-\frac{p}{2}(e^{v_n^2}-1-v_n^2)\right]\,\ud x.
\end{align*}
Observe that there exists $R_p>0$ such that for any $n$,
\begin{align*}
\left\{
\begin{array}{lcr}
\int_{\{|u_n(x)|\ge R_p\}}\left[u_n^2(e^{u_n^2}-1)-\frac{p}{2}(e^{u_n^2}-1-u_n^2)\right]\,\ud x\ge0,\\
\int_{\{|u_nv_n(x)|\ge R_p\}}\left[|u_nv_n|(e^{|u_nv_n|}-1)-\frac{p}{2}(e^{|u_nv_n|}-1-|u_nv_n|^2)\right]\,\ud x\ge0,\\
\int_{\{|v_n(x)|\ge R_p\}}\left[v_n^2(e^{v_n^2}-1)-\frac{p}{2}(e^{v_n^2}-1-v_n^2)\right]\,\ud x\ge0.
\end{array}
\right.
\end{align*}
By the Lebesgue dominated convergence theorem and since $u,v\equiv 0$,
\begin{align*}
&\liminf_{n\rg\iy}K_p(u_n,v_n)\\
&\ge\liminf_{n\rg\iy}\frac{\mu_1}{p}\int_{\{|u_n(x)|\le R_p\}}\left[u_n^2(e^{u_n^2}-1)-\frac{p}{2}(e^{u_n^2}-1-u_n^2)\right]\,\ud x\\
&+\liminf_{n\rg\iy}\frac{2\beta}{p}\int_{\{|u_nv_n(x)|\le R_p\}}\left[|u_nv_n|(e^{|u_nv_n|}-1)-\frac{p}{2}(e^{|u_nv_n|}-1-|u_nv_n|^2)\right]\,\ud x\\
&+\liminf_{n\rg\iy}\frac{\mu_2}{p}\int_{\{|v_n(x)|\le R_p\}}\left[v_n^2(e^{v_n^2}-1)-\frac{p}{2}(e^{v_n^2}-1-v_n^2)\right]\,\ud x\\
&=0.
\end{align*}
Then
$$
\liminf_{n\rg\iy}\int_{\Omega}(|\na u_n|^2+|\na v_n|^2)\,\ud x\le\frac{2p}{p-2}\liminf_{n\rg\iy}I(u_n,v_n)=\frac{2p}{p-2}c_\beta.
$$
Since $p$ is arbitrary, we get that $$\liminf_{n\rg\iy}\int_{\Omega}(|\na u_n|^2+|\na v_n|^2)\,\ud x\le 2c_\beta.$$ Since $c_\beta<E_1+E_2<4\pi$, without loss of generality, we assume that $\sup_n\|\na u_n\|_2^2:=\delta<4\pi$. Let $\ti{u}_n=\frac{u_n}{\|\na u_n\|_2}$ and $q>1$ such that $q\delta<4\pi$. By Lemma \ref{at} and \re{gradient},
\begin{align*}
&\limsup_{n\rg\iy}\int_{\Omega}\left(e^{qu_n^2}-1\right)\,\ud x=\limsup_{n\rg\iy}\int_{\Omega}\left(e^{q\|\na u_n\|_2^2\ti{u}_n^2}-1\right)\,\ud x\\
&\le\frac{C}{1-q\delta/4\pi}\limsup_{n\rg\iy}\|\ti{u}_n\|_2^2=\frac{C}{1-q\delta/4\pi}\frac{1}{\liminf_{n\rg\iy}\|\na u_n\|_2^2}\lim_{n\to \infty}\|u_n\|_2^2=0.
\end{align*}
We have
$$
\lim_{n\rg\iy}\int_{\Omega}u_n^2\left(e^{u_n^2}-1\right)\,\ud x=0\,\,\mbox{and}\,\,\lim_{n\rg\iy}\int_{\Omega}|u_nv_n|\left(e^{|u_nv_n|}-1\right)\,\ud x=0.
$$
Recalling that $u_n\rg0$ in $L^2(\Omega)$ as $n\rg\iy$, and
$$
\int_{\Omega}\left(|\na u_n|^2+\la_1|u_n|^2\right)\,\ud x=\int_{\Omega}\left[\mu_1 u_n^2\left(e^{u_n^2}-1\right)+\beta|u_nv_n|\left(e^{|u_nv_n|}-1\right)\right]\,\ud x,
$$
the following holds
$$
\lim_{n\rg\iy}\|\na u_n\|_2=0,
$$
which contradicts \re{gradient}.

\vskip0.1in
{\bf Case 2.} $u\not\equiv0, v\equiv 0$ or $u\equiv 0, v\not\equiv0$. Assume $u\equiv 0, v\not\equiv0$, then $v$ is a nontrivial solution of problem \re{lm1} with $i=2$.
Then $J_2(v)\ge E_2$. By \re{con} and Lemma \ref{bd1},
$$
\lim_{n\rg\iy}\int_{\Omega}\left(e^{u_n^2}-1-u_n^2\right)\,\ud x=\lim_{n\rg\iy}\int_{\Omega}\left(e^{|u_nv_n|}-1-|v_nu_n|\right)\,\ud x=0
$$
so that
$$
\limsup_{n\rg\iy}I(u_n,v_n)=\frac{1}{2}\limsup_{n\rg\iy}\|\na u_n\|_2^2+J_2(v)=c_\beta
$$
which yields
$$
\limsup_{n\rg\iy}\|\na u_n\|_2^2=2(c_\beta-J_2(v))<2(c_\beta-E_2)<2E_1<4\pi.
$$
Analogously to the Case 1 we can get a contradiction. This completes the proof.
\ep

\s{The case $\beta>0$ large (strong cooperation): proof of Theorem \ref{Th2}}\label{sec4}
\subsection{Nehari manifold} Define the Nehari manifold as follows
$$
\mathcal{N}_{\beta}:=
\left\{
(u,v)\in X\setminus\{(0,0)\}: \ti{J}(u,v):=\langle I'(u,v),(u,v)\rangle=0
\right\},
$$
so that
\begin{align*}
\ti{J}(u,v) =\int_{\Omega}(|\na u|^2+\la_1u^2+|\na v|^2+\la_2v^2)\,\ud x-\int_{\Omega}\left[uH_u(u,v)+vH_v(u,v)\right]\,\ud x,
\end{align*}
and the least energy level
$$
d_\beta:=\inf_{(u,v)\in\mathcal{N}_{\beta}}I(u,v).
$$
Obviously, $(u_1,0)$ and $(0,u_2)$ belongs to $\mathcal{N}_{\beta}$. So $d_\beta\le\min\{E_1,E_2\}$. As in Lemma \ref{lower} one has $d_\beta\ge0$.
\bl\lab{asym}
The following holds:
\begin{itemize}
\item [{\rm(i)}]  For any $\beta>0$, $d_\beta>0$.
\item [{\rm(ii)}]  $\displaystyle d_\beta=\inf_{(u,v)\in X\setminus\{(0.0)\}}\max_{t\ge0}I(\sqrt{t}u,\sqrt{t}v)$, and  $d_\beta$ is non-increasing with respect to $\beta>0$.
\item [{\rm(iii)}]  $d_\beta<\min\{E_1,E_2\}$ if $\beta\ge\bar{\beta}_0.$
\item [{\rm(iv)}]  $d_\beta\rg0$ as $\beta\rg+\iy.$
\end{itemize}
\el

\bp For convenience let us divide the proof into three steps:

{\bf Step 1.} We prove $(i)$ by contradiction. If $d_\beta=0$ for some $\beta>0$, then there exists $(u_n,v_n)\in\mathcal{N}_{\beta}$ such that $I(u_n,v_n)\rg0$, as $n\rg\iy$. As in Lemma \ref{lower}, $\|\na u_n\|_2\rg0$ and $\|\na v_n\|_2\rg0$, as $n\rg\iy$. Moreover, for any $\al\in(0,\mathcal{S}_4)$ and $n$ large enough,
\begin{align}\lab{guji1n}
\al\left[\int_{\Omega}u_n^2\left(e^{u_n^2}-1\right)\,\ud x\right]^{\frac12}\le\int_{\Omega}|\na u_n|^2\,\ud x,\qquad
\al\left[\int_{\Omega}v_n^2\left(e^{v_n^2}-1\right)\,\ud x\right]^{\frac12}\le\int_{\Omega}|\na v_n|^2\,\ud x,
\end{align}
and
\be\lab{vann}
\lim_{n\rg\iy}\int_{\Omega}u_n^2\left(e^{u_n^2}-1\right)=\lim_{n\rg\iy}\int_{\Omega}v_n^2\left(e^{v_n^2}-1\right)\,\ud x=0.
\ee
Since $\la_1,\la_2>-\Lambda_1$, we get
\begin{align}\lab{guji3n}
\begin{cases}
\int_{\Omega}(|\na u_n|^2+\la_1u_n^2)\,\ud x\ge \min\{1,\frac{\la_1+\Lambda_1}{\Lambda_1}\}\int_{\Omega}|\na u_n|^2\,\ud x,\\
\int_{\Omega}(|\na v_n|^2+\la_2v_n^2)\,\ud x\ge \min\{1,\frac{\la_2+\Lambda_1}{\Lambda_1}\}\int_{\Omega}|\na v_n|^2\,\ud x.
\end{cases}
\end{align}
Then, by Lemma \ref{bd1}, \re{guji1n}, \re{guji3n} and $(u_n,v_n)\in\mathcal{N}_{\beta}$, for $n$ large enough,
\begin{align*}
&\left[\int_{\Omega}u_n^2\left(e^{u_n^2}-1\right)\,\ud x\right]^{\frac12}+\left[\int_{\Omega}v_n^2\left(e^{v_n^2}-1\right)\,\ud x\right]^{\frac12}\\
&\le c\left(\left[\int_{\Omega}u_n^2\left(e^{u_n^2}-1\right)\,\ud x\right]^{\frac12}
+\left[\int_{\Omega}v_n^2\left(e^{v_n^2}-1\right)\,\ud x\right]^{\frac12}\right)^2,
\end{align*}
which contradicts \re{vann}, where
$$
c=\frac{1}{\alpha}\max\left\{1,\frac{\Lambda_1}{\lambda_1+\Lambda_1},\frac{\Lambda_1}{\lambda_2+\Lambda_2}\right\}\max\{\mu_1,\mu_2,\beta\}>0.
$$

{\bf Step 2.} We claim that for $\beta>0$,
$$
d_\beta=\inf_{(u,v)\in X\setminus\{(0.0)\}}\max_{t\ge0}I(\sqrt{t}u,\sqrt{t}v).
$$
As a straightforward consequence, $d_\beta$ is non-increasing in $\beta>0$.
In fact, for any $(u,v)\in X\setminus\{(0,0)\}$ and $t\ge0$ and setting $f(t)=I(\sqrt{t}u,\sqrt{t}v)$, namely
\begin{align*}
f(t)=&\frac{t}{2}\int_{\Omega}\left(|\na u|^2+\la_1u^2+|\na v|^2+\la_2v^2\right)\,\ud x\\
&-\frac{1}{2}\int_{\Omega}\left[\mu_1\left(e^{tu^2}-1-tu^2\right)+\mu_2\left(e^{tv^2}-1-tv^2\right)+2\beta\left(e^{t|uv|}-1-t|uv|\right)\right]\,\ud x\ ,
\end{align*}
similarly to Lemma \ref{upper} one has $f(t)>0$ for $t$ small enough,
and $f(t)<0$ for $t$ large enough; moreover, it is easy to check that $f$ admits a unique critical point.
Then, there exists a unique $t_0>0$ such that
$$
f(t_0)=\max_{t\ge0}f(t)>0.
$$
From $f'(t_0)=0$ one has $(\sqrt{t_0}u,\sqrt{t_0}v)\in\mathcal{N}_\beta$. On the other hand, by the uniqueness of critical points of $f$, for any $(u,v)\in\mathcal{N}_\beta$, $f(1)=\max_{t\ge0}f(t)$. Thus
$$
\inf_{(u,v)\in E\setminus\{(0.0)\}}\max_{t\ge0}I(\sqrt{t}u,\sqrt{t}v)=\inf_{(u,v)\in\mathcal{N}_\beta}I(u,v)=d_\beta.
$$

{\bf Step 3.} Assume $\la_2\le\la_1$. For any $t\ge0$, let $g(t)=I(\sqrt{t}u_1,\sqrt{t}u_1)$, where $u_1=u_{\lambda_1,\mu_1}$ is given in Section 1. That is,
$$
g(t)=\frac{t}{2}\int_{\Omega}\left(2|\na u_1|^2+(\la_1+\la_2)u_1^2\right)\,\ud x-\frac{1}{2}(\mu_1+\mu_2+2\beta)\int_{\Omega}\left(e^{tu_1^2}-1-tu_1^2\right)\,\ud x.
$$
As in the previous step, there exists $t_1>0$ such that $g(t_1)=\max_{t\ge0}g(t)>0$ and $(\sqrt{t_1}u_1,\sqrt{t_1}u_1)\in\mathcal{N}_\beta$. Then, $d_\beta \leq I(\sqrt{t_1}u_1,\sqrt{t_1} u_1)<\max_{t\ge0}h(t)$,
where
\begin{equation}\label{h(t)}
h(t)=t\int_{\Omega}\left(|\na u_1|^2+\la_1u_1^2\right)\,\ud x-\beta\int_{\Omega}\left(e^{tu_1^2}-1-tu_1^2\right)\,\ud x.
\end{equation}
Obviously, there exists $t_\beta>0$ such that $\max_{t\ge0}h(t)=h(t_\beta)$ and
$$
h'(t_\beta)=\int_{\Omega}\left(|\na u_1|^2+\la_1u_1^2\right)\,\ud x-\beta\int_{\Omega}u_1^2\left(e^{t_\beta u_1^2}-1\right)\,\ud x=0.
$$
Since
$$
\int_{\Omega}\left(|\na u_1|^2+\la_1u_1^2\right)\,\ud x=\mu_1\int_{\Omega}u_1^2\left(e^{u_1^2}-1\right)\,\ud x,
$$
we get $t_\beta\in(0,\beta_5/\beta)$. So by \re{h(t)},
$$
d_\beta<h(t_\beta)<t_\beta\int_{\Omega}\left(|\na u_1|^2+\la_1u_1^2\right)\,\ud x<4E_1\beta_5/\beta.
$$
In a similar fashion, if $\la_1\le\la_2$, we get that $d_\beta<4E_2\beta_6/\beta$. In conclusion,
$$
d_\beta<\frac{4}{\beta}\max\{E_1\beta_5,E_2\beta_6\},
$$
which implies $d_\beta\rg0$, as $\beta\rg+\iy$, and $d_\beta<\min\{E_1,E_2\}$ provided
\[
\beta\ge\bar{\beta}_0=4\frac{\max\{E_1\beta_5, E_2\beta_6\}}{\min\{E_1,E_2\}}.\qedhere
\]
\ep

\subsection{Mountain pass geometry}
In contrast with the proof of Theorem \ref{Th1}, here, via the mountain pass theorem, we give a Palais-Smale sequence at the least energy level. Now, we show that the functional $I$ satisfies the mountain pass geometry. Obviously, $I(0,0)=0$. For some $(u_0,v_0)\in E\setminus\{(0,0)\}$, reasoning similarly as above, there exists $t_0>0$ such that $I(t_0u_0,t_0v_0)<0.$ Set
$$
\G:=\{\g\in C([0,1], X): \g(0)=0,I(\g(1))<0\};
$$
then $\G\not=\emptyset$, since $\g_0\in\G$, where $\g_0(t)=(tt_0u_0,tt_0v_0)$. For any $\beta>0$, define
$$
d_\beta^{MP}:=\inf_{\g\in\G}\max_{t\in[0,1]}I(\g(t)).
$$

\begin{lemma}
For any $\beta>0$, there exists $\rho^*>0$ such that, for every $\rho\in (0,\rho^*)$,
\[
d_\beta^{MP}\ge\inf_{\stackrel{(u,v)\in X}{\|(u,v)\|=\rho}}I(u,v)>0.\qedhere
\]
\end{lemma}

\bp
For $\al\in(0,\mathcal{S}_4)$ and for any $(u,v)\in X$ with $\|\na u\|_2^2+\|\na v\|_2^2=\rho^2$ with $\rho>0$ sufficiently small, one has
\begin{align*}
\al\left[\int_{\Omega}u^2\left(e^{u^2}-1\right)\,\ud x\right]^{\frac12}\le\int_{\Omega}|\na u|^2\,\ud x,\qquad
\al\left[\int_{\Omega}v^2\left(e^{v^2}-1\right)\,\ud x\right]^{\frac12}\le\int_{\Omega}|\na v|^2\,\ud x,
\end{align*}
which yields
\begin{align*}
\int_{\Omega}\left(e^{u^2}-1-u^2\right)\,\ud x\le\frac{\rho^4}{2\al^2},\qquad
\int_{\Omega}\left(e^{v^2}-1-v^2\right)\,\ud x\le\frac{\rho^4}{2\al^2}.
\end{align*}
Since $\la_1,\la_2>-\Lambda_1$, for $\rho>0$ small enough we get
$$
I(u,v)\ge\frac{1}{2}\min\left\{1,\frac{\la_1+\Lambda_1}{\Lambda_1},\frac{\la_2+\Lambda_1}{\Lambda_1}\right\}\rho^2-
\frac{\rho^4}{4\al^2}(\mu_1+2\beta+\mu_2)>0.
$$
Since $I(\g(1))<0$ for any $\g\in\G$ and $\g(0)=0$, there exists $t^\ast\in(0,1)$ such that $\|\g(t^\ast)\|=\rho$. Therefore,
\[
d_\beta^{MP}\ge\inf_{\stackrel{(u,v)\in X}{\|(u,v)\|=\rho}}I(u,v)>0.\qedhere
\]
\ep
Similarly, for any $(u,v)\in X\setminus\{(0,0)\}$, $I(tu,tv)<0$ for $t>0$ large. Hence $$d_\beta^{MP}\le d_\beta.$$

\subsection{Conclusion of the proof of Theorem \ref{Th2}}
\bp {\bf Step 1.} Since $I$ satisfies the mountain pass geometry, by virtue of the Ekeland variational principle \cite{Ekeland}, there exists $\{(u_n,v_n)\}\subset X$ such that, as $n\rg\iy$
$$
I(u_n,v_n)\rg d_\beta^{MP}\,\,\mbox{and}\,\,\, I'(u_n,v_n)\rg0\,\,\,\mbox{in}\,\,\, X^\ast.
$$
From \re{am} with $p=4$, we have
\be\lab{bound2}
I(u_n,v_n)-\frac{1}{4}\lan I'(u_n,v_n),(u_n,v_n)\ran\ge\frac{1}{4}\int_{\Omega}(|\na u_n|^2+|\na v_n|^2+\la_1u_n^2+\la_2v_n^2)\,\ud x.
\ee
Recalling that $\la_1,\la_2>-\Lambda_1$, we get that $\{(u_n,v_n)\}$ is bounded in $X$. Thus $u_n\rg u_\beta$ and $v_n\rg v_\beta$ weakly in $H_0^1(\Omega)$ and a.e. in $\Omega$, as $n\rg\iy$. From \re{am} with $p=3$, we have
$$
\sup_{n}\int_{\Omega}u_n^2(e^{u_n^2}-1)\,\ud x<\iy,\,\,\,\sup_{n}\int_{\Omega}v_n^2(e^{v_n^2}-1)\,\ud x<\iy.
$$
Thanks to \cite[Lemma 2.1]{FMR} and the fact that $u_n\rg u_\beta$ and $v_n\rg v_\beta$ weakly in $H_0^1(\Omega)$ as $n\rg\iy$, up to a subsequence, there holds that $\lan I'(u_\beta,v_\beta),(\vp,\phi)\ran=0$, for any $\vp,\phi\in C_0^\iy(\Omega)$. Namely, $I'(u_\beta,v_\beta)=0$ in $X^\ast$. Moreover, by Fatou's Lemma,
\begin{align*}
d_\beta^{MP}&=\lim_{n\rg\iy}\left[I(u_n,v_n)-\frac{1}{4}\lan I'(u_n,v_n),(u_n,v_n)\ran\right]\\
&\ge I(u_\beta,v_\beta)-\frac{1}{4}\lan I'(u_\beta,v_\beta),(u_\beta,v_\beta)\ran=I(u_\beta,v_\beta).
\end{align*}
If $(u_\beta,v_\beta)\not=(0,0)$, then $(u_\beta,v_\beta)\in\mathcal{N}_{\beta}$ and
$$
d_\beta^{MP}\ge I(u_\beta,v_\beta)\ge d_\beta.
$$
By $d_\beta^{MP}\le d_\beta$, we get $I(u_\beta,v_\beta)=d_\beta$. Since $I(|u_\beta|,|v_\beta|)=d_\beta$ and $(|u_\beta|,|v_\beta|)\in\mathcal{N}_{\beta}$, by the Lagrange multiplier theorem, there exists $\kappa\in\R$ such that
$$
I'(|u_\beta|,|v_\beta|)=\kappa \ti{J}'(|u_\beta|,|v_\beta|).
$$
As
$$
\lan \ti{J}'(|u_\beta|,|v_\beta|), (|u_\beta|,|v_\beta|)\ran=-4\int_{\Omega}\left(\mu_1 u_{\beta}^4e^{u_{\beta}^2}+2\beta u_{\beta}^2v_{\beta}^2e^{|u_{\beta} v_{\beta}|}+\mu_2v_{\beta}^4e^{v_{\beta}^2}\right)\,\ud x<0,
$$
we have $\kappa=0$. Thus, $I(|u_\beta|,|v_\beta|)=d_\beta$ and $I'(|u_\beta|,|v_\beta|)=0$. By the maximum principle, we have that $(|u_\beta|,|v_\beta|)$ is a positive ground state solution of \re{q1}.
\vskip0.1in
{\bf Step 2.} Next we show that $u_\beta\not=0$ and $v_\beta\not=0$. We argue by contradiction.
\vskip0.1in
{\bf Case 1.} $u_\beta=v_\beta=0$. As done in Section 3.4, we have
$$
\lim_{n\rg\iy}\|\na u_n\|_2=0\,\,\,\mbox{or}\,\,\,\lim_{n\rg\iy}\|\na v_n\|_2=0
$$
and without loss of generality, we may assume that $\|\na u_n\|_2\rg0$, as $n\rg\iy$. Since $d_\beta^{MP}>0$, $\liminf_{n\rg\iy}\|\na v_n\|_2>0$. At the same time we have
$$
d_\beta^{MP}+o_n(1)=I(u_n,v_n)=J_{\la_2,\mu_2}(v_n)+o_n(1)
$$
and
$$
o_n(1)=I'(u_n,v_n)=J_{\la_2,\mu_2}'(v_n)+o_n(1).
$$
Since $\liminf_{n\rg\iy}\|\na v_n\|_2>0$, there exists $t_n>0$ such that $t_n\rg1$, as $n\rg\iy$ and $t_nv_n\in\mathcal{N}_{\la_2,\mu_2}$. Then,
$$
J_{\la_2,\mu_2}(v_n)=J_{\la_2,\mu_2}(t_n v_n)+o_n(1)\ge E_2+o_n(1).
$$
This yields
$$
d_\beta^{MP}\ge E_2,
$$
which contradicts the fact $d_\beta^{MP}\le d_\beta<\min\{E_1,E_2\}$ if $\beta\ge\bar{\beta}_0$ (recall Lemma \ref{asym}-(iii)).
\vskip0.1in
{\bf Case 2.} $u\not\equiv0, v=0$ or $u=0, v\not\equiv0$. Assume that $u=0, v\not\equiv0$, then $J_{\la_2,\mu_2}'(v_\beta)=0$ and $J_{\la_2,\mu_2}(v_\beta)\ge E_2$. Similarly to Section 3.4, one has that $\lim_{n\rg\iy}\|\na u_n\|_2=0$ and hence, following the previous Case 1, we can get a contradiction.
\vskip0.1in
{\bf Step 3.} For any $\beta>\bar{\beta}_0$, let $(u_\beta,v_\beta)$ be a positive ground state solution to \re{q1}. Similarly to \re{bound2}, thanks to $\la_1,\la_2>-\Lambda_1$, we have
$$
\frac{1}{4}\min\left\{1,\frac{\la_1+\Lambda_1}{\Lambda_1},\frac{\la_2+\Lambda_1}{\Lambda_1}\right\}\|(u_\beta,v_\beta)\|^2\le d_\beta.
$$
From Lemma \ref{asym}-(iv) we have $d_\beta\to 0$; thus, $u_\beta\rg0$ and $v_\beta\rg0$ strongly in $H_0^1(\Omega)$ as $n\rg\iy$. This completes the proof.
\ep

\s{The case $\beta<0$ small (weak competition): proof of Theorem \ref{Th3}.}\label{sec5}

In this section, we are concerned with positive solutions of \re{q1} in the repulsive case $\beta<0$. For this purpose, the associated functional is given by
$$
J(u,v)=\frac12\int_{\Omega}(|\na u|^2+|\na v|^2+\la_1u^2+\la_2v^2)-\int_{\Omega}\ti{H}(u,v),\,\,(u,v)\in X,
$$
where
$$
\ti{H}(u,v)=\frac{\mu_1}{2}\ti{G}(u,u)+\beta \ti{G}(u,v)+\frac{\mu_2}{2}\ti{G}(v,v),
$$
and $$\ti{G}(u,v)=e^{u_+v_+}-1-u_+v_+,\,\,\, u_+=\max\{u,0\}, v_+=\max\{v,0\}.$$
It is standard to prove that $J$ is of $C^1$-class and that critical points $(u,v)$ turn out to be weak solutions of the system
\begin{align*}
\begin{cases}
-\DD u+\la_1u=\mu_1u_+\left(e^{u_+^2}-1\right)+\beta v_+\left(e^{u_+v_+}-1\right)& \mbox{in}\,\,\,\Omega,\\
-\DD v+\la_2v=\mu_2v_+\left(e^{v_+^2}-1\right)+\beta u_+\left(e^{u_+v_+}-1\right) & \mbox{in}\,\,\,\Omega,\\
u,v\in H_0^1(\Omega).
\end{cases}
\end{align*}
Let $u_-=\max\{-u,0\}$ and take $(u_-,0)$ as a test function to get that $u$ is non-negative (observe that $u_-v_+\left(e^{u_+v_+}-1\right)=0$). Then we have
$$
-\DD u\le\left[\mu_1\left(e^{u^2}-1\right)+|\la_1|\right]u,\,\,\,x\in\Omega.
$$
Since $e^{u^2}\in L^t(\Omega)$ for any $t>0$, by the Nash-Moser iteration technique (see \cite{Gongbao} and also \cite{ZJM}), one has $u\in L^\iy(\Omega)$. Similarly, $v$ is nonnegative and $v\in L^\iy(\Omega)$. Let
\begin{align*}
c(x)=\left\{
\begin{array}{ll}
|\beta|v(x)\frac{e^{u(x)v(x)}-1}{u(x)},&\,\,\,\mbox{if}\,\,\,u(x)\not=0,\\
0,&\,\,\,\mbox{if}\,\,\,u(x)=0,
\end{array}
\right.
\end{align*}
then $c\in L^\iy(\Omega)$ and
$$
-\DD u+(c(x)+\la_1)u=\mu_1u\left(e^{u^2}-1\right)\ge0,\,\,x\in\Omega.
$$
By virtue of the weak Harnack inequality for supersolutions (see \cite[Theorem 8.18]{Tru}), $u$ is positive and the same holds for $v$.

In what follows, we borrow some ideas from J. Byeon and L. Jeanjean \cite{byeon2} (see also S. Kim \cite{Kim}) to investigate the existence of positive solutions to \re{q1} for $\beta$ slightly negative.

\subsection{Energy estimate}
Let
$$
\mathcal{S}:=\{(u,v): u\in S_{\la_1,\mu_1}, v\in S_{\la_2,\mu_2}\},
$$
where $S_{\la_i,\mu_i}$, $i=1,2$, are given in Section \ref{subsec:limitproblem}, then $\mathcal{S}\not=\emptyset$ and it is compact in $X$ by Lemma \ref{priori}. For fixed $u_{\la_1,\mu_1}\in S_{\la_1,\mu_1}$ and $u_{\la_2,\mu_2}\in S_{\la_2,\mu_2}$, let
$$\g_0(s,t)=(su_{\la_1,\mu_1}, tu_{\la_2,\mu_2})\in X,\,\,s,t\ge0.$$
By $\beta\in(-\sqrt{\mu_1\mu_2},0)$ and Lemma \ref{bd1}, for any $x,y\ge0$, we have
$$
\ti{H}(x,y)\ge \left(1+\frac{\beta}{\sqrt{\mu_1\mu_2}}\right)\left[\mu_1\ti{G}(x,x)+\mu_2\ti{G}(y,y)\right]\ge0.
$$
Moreover, since $\ti{G}(x,y)\ge x^2y^2/2$ for any $x,y\ge0$, then we have
$$
\ti{H}(x,y)\ge\left(1+\frac{\beta}{\sqrt{\mu_1\mu_2}}\right)\frac{\mu_1x^4+\mu_2y^4}{2},\,\,x,y\ge0.
$$
It follows that for any $s,t\ge0$,
\begin{align*}
J(\g_0(s,t))\le&\frac{s^2}{2}\int_{\Omega}(|\na u_{\la_1,\mu_1}|^2+\la_1u_{\la_1,\mu_1}^2)+
\frac{t^2}{2}\int_{\Omega}(|\na u_{\la_2,\mu_2}|^2+\la_2u_{\la_2,\mu_2}^2)\\
&-\frac14\left(1+\frac{\beta}{\sqrt{\mu_1\mu_2}}\right)\int_{\Omega}(s^4u_{\la_1,\mu_1}^4+t^4u_{\la_2,\mu_2}^4)\rightarrow-\iy,\,\,\mbox{as}\,\,s^2+t^2\rg\iy.
\end{align*}
By Lemma \ref{atn} there exist $s_0>1$ and $t_0\in(0,1)$ such that $J(\g_0(s_0,s_0))<0$ and for $i=1,2$,
\be\lab{brower2}
\lan J_{\la_i,\mu_i}'(t_0u_{\la_i,\mu_i}),t_0u_{\la_i,\mu_i}\ran>0,\,\,\lan J_{\la_i,\mu_i}'(s_0u_{\la_i,\mu_i}),s_0u_{\la_i,\mu_i}\ran<0.
\ee

\noindent Set
$$
\ti{d}_\beta:=\max_{t,s\in[t_0,s_0]}J(\g_0(s,t)),
$$
then it is straightforward to show that $\ti{d}_\beta\ge c>0$ for $\beta$ sufficiently small, where $c$ is independent of $\beta$. In particular, we obtain the asymptotic behavior of $\ti{d}_\beta$ in the following
\bl
$$
\ti{d}_\beta\rightarrow E_{\la_1,\mu_1}+E_{\la_2,\mu_2},\,\,\mbox{as}\,\,\beta\rg0.
$$
\el
\bp
Noting that $u_{\la_1,\mu_1}, u_{\la_2,\mu_2}\in L^\iy(\Omega)$, we have for $t,s\in[t_0,s_0]$,
$$
J(\g_0(s,t))=J_{\la_1,\mu_1}(su_{\la_1,\mu_1})+J_{\la_2,\mu_2}(u_{\la_2,\mu_2})+O(\beta),\,\,\mbox{as}\,\,\beta\rg0.
$$
Recalling that
$$
E_{\la_1,\mu_1}=J_{\la_1,\mu_1}(u_{\la_1,\mu_1})=\max_{s\in[t_0,s_0]}J_{\la_1,\mu_1}(su_{\la_1,\mu_1})
$$
and
$$
E_{\la_2,\mu_2}=J_{\la_2,\mu_2}(u_{\la_2,\mu_2})=\max_{t\in[t_0,s_0]}J_{\la_2,\mu_2}(tu_{\la_2,\mu_2}),
$$
we have, as $\beta\rg0$,
\[
\max_{t,s\in[t_0,s_0]}J(\g_0(s,t))=\max_{t,s\in[t_0,s_0]}J_{\la_1,\mu_1}(su_{\la_1,\mu_1})+ \max_{t,s\in[t_0,s_0]}J_{\la_2,\mu_2}(tu_{\la_2,\mu_2})]+O(\beta).   \qedhere
\]
\ep

\subsection{A minimax value}

For any $\delta>0$, let
\begin{align*}
\mathcal{S}^\delta:=
\left\{(u,v)\in X\Big|
\begin{array}{ll}
\hbox{$(u,v)=(\ti{u},\ti{v})+(\bar{u},\bar{v})$,}\quad \hbox{$(\ti{u},\ti{v})\in\mathcal{S}, \|(\bar{u},\bar{v})\|\le\delta$}
\end{array}
\right\},
\end{align*}
be the $\delta$-neighborhood of $\mathcal{S}$. Then $\mathcal{S}^\delta$ is bounded in $X$.
\bl\lab{unb} For $\delta>0$ sufficiently small, we have
$$
\sup_{(u,v)\in\mathcal{S}^\delta}\int_{\Omega}(u^2e^{u^2}+v^2e^{v^2})<\iy.
$$
\el
\bp
If not, there exist $\{\delta_n\}$ and $(u_n,v_n)\in\mathcal{S}^{\delta_n}$, such that as $n\rg\iy$, $\delta_n\rg0$ and
$$
\int_{\Omega}(u_n^2e^{u_n^2}+v_n^2e^{v_n^2})\rg+\iy.
$$
Then by the compactness of $\mathcal{S}$, without loss of generality, for some $(u,v)\in\mathcal{S}$, we have $u_n\rg u$ and $v_n\rg v$ strongly in $H_0^1(\Omega)$. In the following, we show that
$$
\lim_{n\rg\iy}\int_{\Omega}(u_n^2e^{u_n^2}+v_n^2e^{v_n^2})=\int_{\Omega}(u^2e^{u^2}+v^2e^{v^2}).
$$
Suppose for the moment that this hold true, then we reach a contraction thanks to the fact $e^{\al u^2}, e^{\al v^2}\in L^1(\Omega)$ for any $\al>0$, and $u,v\in L^\infty(\Omega)$.

\noindent It is enough to show that
$$
\lim_{n\rg\iy}\int_{\Omega}u_n^2e^{u_n^2}=\int_{\Omega}u^2e^{u^2}.
$$
For any $\e>0$, there exists $C_\e>0$ such that
$$
2|x^2-y^2|\le\e x^2+C_\e|x-y|^2,\,\,x,y\in\R.
$$
Then, for some $C>0$,
\begin{align*}
\int_{\Omega}|u_n^2e^{u_n^2}-u^2e^{u^2}|&\le\int_{\Omega}e^{u^2}u_n^2[e^{|u_n^2-u^2|}-1]+\int_{\Omega}e^{u^2}|u_n^2-u^2|\\
&\le \left(\int_{\Omega}e^{4u^2}\right)^{1/4}\left(\int_{\Omega}u_n^8\right)^{1/4}\left(\int_{\Omega}[e^{2|u_n^2-u^2|}-1]\right)^{1/2}\\
&\,\,\,\,\,\,+\left(\int_{\Omega}e^{2u^2}\right)^{1/2}\left(\int_{\Omega}|u_n^2-u^2|^2\right)^{1/2}\\
&\le C\left(\int_{\Omega}e^{\e u^2}[e^{C_\e|u_n-u|^2}-1]+\int_{\Omega}[e^{\e u^2}-1]\right)^{1/2}+o_n(1).
\end{align*}
By Lemma \ref{at}, from $\|\na(u_n-u)\|_2\rg0$, as $n\rg\iy$ one has
\begin{align*}
&\limsup_{n\rg\iy}\int_{\Omega}e^{\e u^2}[e^{C_\e|u_n-u|^2}-1]\\
&\le\limsup_{n\rg\iy}\left(\int_{\Omega}e^{2\e u^2}\right)^{1/2}\left(\int_{\Omega}\left[e^{2C_\e|u_n-u|^2}-1\right]\right)^{1/2}=0.
\end{align*}
Then
$$
\limsup_{n\rg\iy}\int_{\Omega}|u_n^2e^{u_n^2}-u^2e^{u^2}|\le C\left(\int_{\Omega}[e^{\e u^2}-1]\right)^{1/2}.
$$
This concludes the proof as $\e$ is arbitrary.
\ep

\noindent Thanks to the fact that $\g_0(1,1)\in\mathcal{S}$, for $\delta>0$ given above, there exists $0<\tau<\min\{1-t_0,s_0-1\}$ such that
$$
\g_0(s,t)\in\mathcal{S}^{4\delta},\,\,\mbox{if}\,\, |t-1|<\tau,\,|s-1|<\tau,
$$
and
$$
|t-1|<\tau,\,|s-1|<\tau,\,\,\mbox{if}\,\,\g_0(s,t)\in\mathcal{S}^{2\delta}.
$$
In the following, we define a minimax value. Let
\begin{align*}
\G:=
\left\{\g(s,t)\in C([t_0,s_0]^2, X)\Big|
\begin{array}{ll}
\hbox{$\g(s,t)\in\mathcal{S}^{4\delta}$, if $(s,t)\in(1-\tau,1+\tau)^2$,}\\
\hbox{$\g(s,t)=\g_0(s,t)$, if $(s,t)\not\in(1-\tau,1+\tau)^2$}
\end{array}
\right\};
\end{align*}
then obviously, $\g_0\in\G$, which is nonempty. Define
$$
\ti{c}_\beta:=\inf_{\g\in\G}\max_{s,t\in[t_0,s_0]}J(\g(s,t)),
$$
for which we know that for any $\beta$, $\ti{c}_\beta\le\ti{d}_\beta$. This implies
$$
\limsup_{\beta\rg0}\ti{c}_\beta\le E_{\la_1,\mu_1}+E_{\la_2,\mu_2}.
$$
Actually the following holds
\bl\lab{ca} It holds $
\lim_{\beta\rg0}\ti{c}_\beta=E_{\la_1,\mu_1}+E_{\la_2,\mu_2}.
$
\el
\bp
It is enough to show that for any $\g\in\G$,
$$
\max_{s,t\in[t_0,s_0]}J(\g(s,t))\ge E_{\la_1,\mu_1}+E_{\la_2,\mu_2}.
$$
Let $\g(s,t)=(\g_1(s,t),\g_2(s,t))$ and $T_{\g}: [t_0,s_0]^2\longrightarrow\R^2 $ be given by
$$
T_{\g}(s,t)=\left(\lan J_{\la_1,\mu_1}'(\g_1(s,t)),\g_1(s,t)\ran, \lan J_{\la_2,\mu_2}'(\g_2(s,t)),\g_2(s,t)\ran\right).
$$
Observe that $T_{\g}(s,t)=T_{\g_0}(s,t)$ if $(s,t)\in\pl([t_0,s_0]^2)$. By \re{brower2}, $T_{\g}(s,t)\not=(0,0)$ if $(s,t)\in\pl([t_0,s_0]^2)$ and the Brouwer degree
$$
\mbox{deg}_B(T_\g,[t_0,s_0]^2,(0,0))=\mbox{deg}_B(T_{\g_0},[t_0,s_0]^2,(0,0)).
$$
Let $T_i: [t_0,s_0]\longrightarrow\R $, $i=1,2$ be given by
$$
T_1(s):=\lan J_{\la_1,\mu_1}'(su_{\la_1,\mu_1}),su_{\la_1,\mu_1}\ran,\,\,T_2(t):=\lan J_{\la_2,\mu_2}'(tu_{\la_2,\mu_2}),tu_{\la_2,\mu_2}\ran;
$$
then $T_{\g_0}(s,t)=(T_1(s),T_2(t))$ and
\begin{align*}
&\mbox{deg}_B(T_{\g_0},[t_0,s_0]^2,(0,0))=\mbox{deg}_B(T_1,[t_0,s_0],0)\cdot\mbox{deg}_B(T_2,[t_0,s_0],0)=(-1)\cdot(-1)=1.
\end{align*}
It yields that $\mbox{deg}_B(T_\g,[t_0,s_0]^2,(0,0))=1$ and $T_\g(s^\ast,t^\ast)=(0,0)$ for some $(s^\ast,t^\ast)\in(s_0,t_0)^2$. Then
$\g_i(s^\ast,t^\ast)\in\mathcal{N}_{\la_i,\mu_i}$ for $i=1,2$. Since $\beta<0$, one has
\begin{align*}
&\max_{s,t\in[t_0,s_0]}J(\g(s,t))\ge J(\g(s^\ast,t^\ast))\\
&\ge J_{\la_1,\mu_1}(\g_1(s^\ast,t^\ast))+J_{\la_2,\mu_2}(\g_2(s^\ast,t^\ast))\ge E_{\la_1,\mu_1}+E_{\la_2,\mu_2}.  \qedhere
\end{align*}
\ep

\bl\lab{deform}
For any $\delta>0$, there exists $\sigma>0$ such that for $\beta$ small,
$$
\g_0(s,t)\in\mathcal{S}^{\delta/2}\,\,\,\mbox{if}\,\,\, J(\g_0(s,t))\ge \ti{c}_\beta-\sigma.
$$
\el
\bp
Since $J(\g_0(s,t))=J_{\la_1,\mu_1}(s u_{\la_1,\mu_1})+J_{\la_2,\mu_2}(t u_{\la_2,\mu_2})+O(\beta)$ and
$$
J_{\la_1,\mu_1}(s u_{\la_1,\mu_1})+J_{\la_2,\mu_2}(t u_{\la_2,\mu_2})<E_{\la_1,\mu_1}+E_{\la_2,\mu_2}\,\,\,\mbox{if}\,\, (s,t)\not=(1,1),
$$
by Lemma \ref{ca}, there exists $\sigma>0$ such that for $\beta$ small and $J(\g_0(s,t))\ge \ti{c}_\beta-\sigma$, we have $\|\g_0(s,t)-\g_0(1,1)\|<\delta/2$.
\ep
\subsection{Palais-Smale sequence at level $\ti{c}_\beta$}

Choose
\begin{equation}\label{choose_delta}
0<\delta<\min\left\{\rho_{\la_1,\mu_1},\rho_{\la_2,\mu_2}\right\},
\end{equation}
where $\rho_{\la_i,\mu_i}$, $i=1,2$ were given in the proof of Lemma \ref{priori}, namely in \eqref{eq:rho_lamu}.

\bo\lab{bo5} Let $\{\beta_n\}_{n=1}^\infty$ be such that $\lim_{n\rg
\infty}\beta_n=0$ and for all $n$, $\{(u_n,v_n)\}\subset \mathcal{S}^\delta$
with
$$
\lim_{n\rg\infty}J(u_n,v_n)\le E_{\la_1,\mu_1}+E_{\la_2,\mu_2}\ \mbox{and}\
\lim_{n\rg\infty}\na J(u_n,v_n)=0.
$$
Then for $\delta$ sufficiently small, there exists $(u_0,v_0)\in \mathcal{S}$ such that, up to
a subsequence, $u_n\rg u_0$ and $v_n\rg v_0$ strongly in $H_0^1(\Omega)$, as $n\rg\iy$.
\eo

\bp
For $\{(u_n,v_n)\}\subset \mathcal{S}^\delta$ given above, by Lemma \ref{unb}, for $d$ small, we know that
$$
\lim_{n\rg\iy}J(u_n,v_n)=\lim_{n\rg\iy}[J_{\la_1,\mu_1}(u_n)+J_{\la_2,\mu_2}(v_n)]\le E_{\la_1,\mu_1}+E_{\la_2,\mu_2}.
$$
Let $(u_n,v_n)=(\ti{u}_n,\ti{v}_n)+(\bar{u}_n,\bar{v}_n)$, where $(\ti{u}_n,\ti{v}_n)\in\mathcal{S}$ and $(\bar{u}_n,\bar{v}_n)\in X$ with $\|(\bar{u}_n,\bar{v}_n)\|\le \delta$ for all $n$. By the compactness of $\mathcal{S}$, there exist $(\ti{u}_0,\ti{v}_0)\in\mathcal{S}$ and $(\bar{u}_0,\bar{v}_0)\in X$ with $\|(\bar{u}_0,\bar{v}_0)\|\le \delta$, such that, up to a subsequence, $(\ti{u}_n,\ti{v}_n)\rg (\ti{u}_0,\ti{v}_0)$ strongly in $X$ and $(\bar{u}_n,\bar{v}_n)\rg (\bar{u}_0,\bar{v}_0)$ weakly in $X$ as $n\rg\iy$. Obviously, $(u_n,v_n)\rg (u_0,v_0)$ weakly in $X$ as $n\rg\iy$, where $(u_0,v_0)=(\ti{u}_0,\ti{v}_0)+(\bar{u}_0,\bar{v}_0)$. By Lemma \ref{unb} and \cite[Lemma 2.1]{FMR},
$$
J_{\la_1,\mu_1}(u_n)=J_{\la_1,\mu_1}(u_0)+\frac{1}{2}\int_{\Omega}|\na(u_n-u_0)|^2+o_n(1),
$$
and
$$
J_{\la_2,\mu_2}(v_n)=J_{\la_2,\mu_2}(v_0)+\frac{1}{2}\int_{\Omega}|\na(v_n-v_0)|^2+o_n(1).
$$
We also have $\na J_{\la_1,\mu_1}(u_0)=\na J_{\la_2,\mu_2}(v_0)=0$ in $(H_0^1(\Omega))^{\ast}$. By the choice of $\delta$, $u_0\not\equiv0, v_0\not\equiv0$. Then
$$
\lim_{n\rg\iy}J(u_n,v_n)\ge E_{\la_1,\mu_1}+E_{\la_2,\mu_2}+\frac{1}{2}\lim_{n\rg\iy}\int_{\Omega}(|\na(u_n-u_0)|^2+|\na(v_n-v_0)|^2),
$$
which yields $u_n\rg u_0$ and $v_n\rg v_0$ strongly in $H_0^1(\Omega)$, as $n\rg\iy$. Moreover, $u_0\in S_{\la_1,\mu_1}$ and $v_0\in S_{\la_2,\mu_2}$. That is, $(u_0,v_0)\in \mathcal{S}$.
\ep

\noindent For any $c\in\R$, set
$$
J^c:=\{(u,v)\in X: J(u,v)\le c\}.
$$
By Proposition \ref{bo5} and arguing as in \cite[Proposition 4]{byeon2}, there exist $\delta>0$ small and $\omega>0, \beta_0>0$ such that
\begin{center}
$\|\na J(u,v)\|\ge\omega$ for any $u\in J^{\ti{d}_\beta}\bigcap(\mathcal{S}^\delta\setminus\mathcal{S}^{\frac{\delta}{2}})$ and $|\beta|\in(0,\beta_0)$.
\end{center}
Thanks to Lemma \ref{deform}, for any $|\beta|$ sufficiently small, we obtain a bounded Palais-Smale sequence for $J$, by reasoning as in \cite[Proposition 7]{byeon2} (see also \cite{Kim}), namely the following holds

\bo\lab{bo7} Let $\delta>0$ be as in \eqref{choose_delta}. For any $|\beta|$ small, there exists
$\{(u_n,v_n)\}_n\subset J^{\ti{d}_\beta}\cap \mathcal{S}^\delta$ such that
$\|\na J(u_n,v_n)\|\rg 0$, as $n\rg\iy$.
\eo

\subsection{Conclusion of the proof of Theorem \ref{Th3}}
We are now in the position of proving Theorem \ref{Th3}.

\noindent Let $\{(u_n,v_n)\}_n$ be given as in Proposition \ref{bo7}, namely $\{(u_n,v_n)\}\subset\mathcal{S}^\delta$ with $J(u_n,v_n)\le\ti{d}_\beta$ and $\|\na J(u_n,v_n)\|\rg 0$, as $n\rg\iy$. Then, $\{(u_n,v_n)\}$ is bounded in $X$ and weakly converges to some $(u_\beta,v_\beta)\in\mathcal{S}^\delta$. Moreover, by the choice of $\delta$ one has also that $u_\beta\not\equiv0$ and $v_\beta\not\equiv0$. For any $\vp\in C_0^\iy(\Omega)$, as $n\rg\iy$ we have
$$
\int_{\Omega}\left[\na u_n\na\vp+\la_1 u_n\vp-\mu_1(u_n)_+\left(e^{(u_n)_+^2}-1\right)\vp-\beta (v_n)_+\Big(e^{(u_n)_+(v_n)_+}-1\Big)\vp\right]\rg0.
$$
By Lemma \ref{unb} and \cite[Lemma 2.1]{FMR},
$$
-\DD u_\beta+\la_1u_\beta=\mu_1(u_{\beta})_+\left(e^{(u_\beta)_+^2}-1\right)+\beta (v_\beta)_+\Big(e^{(u_\beta)_+(v_\beta)_+}-1\Big)\,\,\,\mbox{in}\,\,\,\Omega,\\
$$
Similarly,
$$
-\DD v_\beta+\la_2v_\beta=\mu_2(v_{\beta})_+\left(e^{(v_\beta)_+^2}-1\right)+\beta (u_\beta)_+\Big(e^{(u_\beta)_+(v_\beta)_+}-1\Big)\,\,\,\mbox{in}\,\,\,\Omega.\\
$$
Thus, $(u_\beta,v_\beta)$ is a positive solution to \re{q1}.

\noindent Finally, it remains to prove the asymptotic behavior of $(u_\beta,v_\beta)$, as $\beta\rg0$. By Lemma \ref{unb}, \cite[Lemma 2.1]{FMR} and the lower semi-continuity of the norm, $J(u_\beta,v_\beta)\le\ti{d}_\beta$. Therefore, as in Proposition \ref{bo5}, for $\delta>0$ small, there exists $(u_0,v_0)\in \mathcal{S}$ such that, up to a subsequence, $(u_\beta, v_\beta)\rg (u_0, v_0)$ strongly in $X$, as $\beta\rg0$.

\medbreak

\noindent{\bf Acknowledgements}\,\,  Part of this work was done while the first and third named authors were visiting Universidade de Lisboa for which they thank the members of CMAFcIO for their warm hospitality. The authors join to thank \textit{Sergio Cacciatori}, with whom they are in debt for having pointed out, during several fruitful discussions, relevant connections of this work to new aspects of the Physics involved.

\end{document}